\documentclass[11pt,a4paper]{article}
\usepackage{fancyhdr}
\usepackage{amsmath}
\usepackage{amsthm}
\usepackage{amssymb}
\usepackage{wasysym}
\usepackage{paralist}
\usepackage{makeidx}
\usepackage[all]{xy}
\usepackage{mathdots}
\usepackage{yhmath}
\usepackage{mathtools}

\usepackage{graphicx}
\usepackage{epstopdf}

\input xy

\newcommand{\nc}{\newcommand}
 
 \nc{\cl}{\centerline}
 \renewcommand{\l}{{\rm len}}
 \nc{\SL}{{\rm SL}}
 \nc{\hatQ}{{\hat Q}}
 \nc{\sgn}{{\rm sgn}}
 
 \newcommand{\id}{{\rm id}}
 
 \nc{\hatlambda}{{\hat\lambda}}
 
 \nc{\daggerlambda}{{\lambda^\dagger}}

 \newcommand{\barS}{{\bar S}}
 \newcommand{\barH}{{\overline H}}
  \newcommand{\barJ}{{\bar J}}

  \newcommand{\dotG}{{\dot{G}}}
    \newcommand{\dotL}{{\dot{L}}}

  \newcommand{\dotE}{{\dot{E}}}
  
  \newcommand{\dotH}{{\dot{H}}}
  
  \newcommand{\dotF}{{\dot{F}}}

  \newcommand{\resp}{{resp}}

  \newcommand{\Hec}{{\rm Hec}}

\nc\diag{{\rm diag}}
\renewcommand{\vert}{{\,|\,}}
\nc{\hatL}{{\hat L}}

\nc{\barE}{{\bar   E}}
\nc{\D}{{\mathcal D}}
\nc{\E}{{\mathcal E}}
\nc{\F}{{\mathcal F}}
\nc{\even}{{\rm e}}
\nc{\ep}{\epsilon}
\nc{\odd}{{\rm o}}
\nc{\Coker}{{\rm Coker}}
\nc{\olE}{{\overline E}}
\nc{\indBG}{{\rm ind}_B^G\,}
\nc{\indHG}{{\rm ind}_H^G\,}

\nc{\que}{{\mathbb Q}}
\nc{\barlambda}{{\bar\lambda}}
\nc{\barmu}{{\bar\mu}}
\nc{\barnu}{{\bar\nu}}
\nc{\bartau}{{\bar\tau}}
\nc{\barm}{{\bar m}}
\nc{\divind}{{\rm div.ind}}
\nc{\tl}{{\tilde{\lambda}}}

\nc{\Sym}{{\rm \Sigma}}
\nc{\Symm}{{\rm Sym}}
\renewcommand{\dim}{{\rm dim\,}}
\newcommand{\q}{\quad}
\newcommand{\de}{\delta}

\newcommand{\nat}{{\mathbb N}}
\newcommand{\Mod}{{\rm Mod}}
\renewcommand{\mod}{{\rm mod}}
\newcommand{\iso}{\cong}

\newcommand{\Sp}{{\rm Sp}}

\newcommand{\bs}{\bigskip}

\renewcommand{\vert}{\,|\,}

\renewcommand{\sgn}{{\rm sgn}}

\newcommand{\reg}{{\rm reg}}

\newcommand{\pol}{{\rm pol}}

\xyoption{all}

\newcommand{\ind}{{\rm ind}}
\renewcommand{\vert}{\,|\,}
 
\newcommand{\zed}{{\mathbb Z}}
\newcommand{\Ext}{{\rm Ext}}
\newcommand{\End}{{\rm End}}
\newcommand{\Hom}{{\rm Hom}}
\newcommand{\cf}{{\rm cf}}

\renewcommand{\mod}{{\rm mod}}

\newcommand{\GL}{{\rm GL}}

\newcommand{\barM}{{\overline  M}}
\newcommand{\res}{{\rm{res\,}}}
\newcommand{\m}{{\rm{Mull}}}
\newcommand{\Mull}{{\rm Mull}}
\renewcommand{\mod}{{\rm{mod}}}

\nc{\geom}{{\rm geom}}
\nc{\rep}{{\rm rep}}

\newcommand{\core}{{\rm core}}

\newcommand{\ch}{{\rm ch\,}}

\renewcommand{\P}{{\mathcal  P}}

\newtheorem{definition}{Definition}[section]
\newtheorem{proposition}[definition]{Proposition}
\newtheorem{theorem}[definition]{Theorem}
\newtheorem{lemma}[definition]{Lemma}
\newtheorem{corollary}[definition]{Corollary}

\newtheorem{remark}[definition]{Remark}

\begin{document}



\centerline{\bf Composition Factors of Tensor Products}
\centerline{\bf of Truncated Symmetric Powers}

\bigskip

\centerline{Stephen Donkin  and Haralampos Geranios}

\bigskip

{\it Department of Mathematics, University of York, York YO10 5DD}\\

\medskip

{\tt stephen.donkin@york.ac.uk,  haralampos.geranios@york.ac.uk}

\bs

\centerline{10  July 2015}
\bs\bs\bs

\section*{Abstract}

\q    Let $G$ be the general linear group of degree $n$ over an algebraically closed field $K$ of characteristic $p>0$.  We study the $m$-fold  tensor product $\barS(E)^{\otimes m}$ of the truncated symmetric algebra $\barS(E)$ of the symmetric algebra $S(E)$ of the natural module $E$ for $G$. We are particularly interested in the set of partitions $\lambda$  occurring as the highest weight of a composition factor of 
$\barS(E)^{\otimes m}$.  We explain how the determination of these composition factors   is related to the determination of the set of composition factors of the $m$-fold tensor product $S(E)^{\otimes m}$ of the symmetric algebra.  We give a complete description of the  composition  factors of $\barS(E)^{\otimes m}$ in terms of \lq \lq distinguished" partitions.  

\q Our main interest is in the classical case, but since the quantised version is essentially no more difficult  we express our results in the general context throughout.

\section*{Introduction}

\bs

\q Let $K$ be an algebraically closed field of characteristic $p>0$.    The problem of finding the irreducible characters of a connected reductive group $G$ over $K$ is one of the main problems of representation theory. In characteristic $0$ the solution to this problem is enshrined in 
Weyl's character formula (see e.g. \cite{Jan},  II, Chapter 5) and for the general linear group in the theory of Schur symmetric functions (see e.g. \cite{EGS}, , Section 3.5). 

\q The problem in positive characteristic is often formulated in terms of the determination of decomposition numbers, i.e, the determination of the multiplicity  of the simple module $L(\mu)$, of highest weight $\mu$, as a composition factor of  the induced module $\nabla(\lambda)$, of highest weight $\lambda$. 

\q In this paper we concentrate on the case $G={\rm GL}_n(K)$, the general linear group of degree $n$, over $K$. We are interested in the tensor product $S^\lambda(E)=S^{\lambda_1}(E) \otimes \cdots \otimes S^{\lambda_m}(E)$, of symmetric powers of the natural module $E$. For fixed degree $r$, the formal characters of the modules $S^\lambda(E)$ are related to the formal characters of the $\nabla(\mu)$ (the Schur symmetric functions) by a certain  known  unitriangular matrix (the transpose of the Kostka matrix, see e.g., \cite{Mac},  Section 6, Table 1, entry $(2,4)$).    Hence,  the decomposition number problem would be solved if one could determine the composition factor multiplicities of the modules $S^\lambda (E)$. So this is a very important (and of course difficult) problem.   Here,  and in related work, we address the  problem of determining the set of composition factors of  $S^\lambda(E)$.

\q Let $m$ be a positive integer. Our method is to analyse first the tensor product of  $m$ truncated symmetric powers and then to use this to analyse  the tensor product of  $m$ symmetric powers. Here we  give an exposition of the general approach via the truncated symmetric powers.   For  general $m$, we give a complete list of the composition factors of   $\barS(E)^{\otimes m}$ in terms of  \lq \lq distinguished" partitions.  One consequence of this description is that this list is also the list of composition factors $L(\lambda)$ of $S(E)^{\otimes m}$ for partitions $\lambda$ with first part at most $m(p-1)$. In particular, for $m=n$, the composition factors of ${\barS}(E)^{\otimes n}$ are the partitions of length at most $n$ with first part at most $n(p-1)$.

\q As an immediate application of our  approach  we recover the  tensor product theorem of Krop, \cite{K} and Sullivan, \cite{Sull}.     This describes the composition factors of the symmetric powers  of the natural module.   Further, in our companion paper, \cite{DG3},  we  obtain a direct analogue for the composition factors of a tensor product of two symmetric powers,  \cite{DG3},  Theorem 4.6. 

\q Moreover, with the  methods used here and in \cite{DG3},  we obtain an application to the representation theory of the symmetric groups:    specifically we 
 determine which irreducible modules occur as composition factors of  Specht modules labelled by   partitions   with   third row length at most one,  \cite{DG3},  Corollary 2.11.

 \q Apart from its relevance to the modular character  problem  we have some other motivation for the  consideration of  tensor products of symmetric powers coming  from our earlier work.   In \cite{DG2} we studied the problem of which polynomial injective modules are injective on restriction to the first infinitesimal subgroup $G_1$    and we gave a solution to this problem  in terms of the \lq\lq index of divisibility" of a polynomially injective module,  \cite{DG2},  Theorem 4.1.    The divisibility index,  in turn,  is determined by the set of  composition factors of $S(E)^{\otimes (n-1)}$, \cite{DG3}, Lemma 3.9.  
 
 \q An explicit solution to the problem of finding all polynomially and infinitesimally injective modules would also resolve the  sticking point of  the paper by De Visscher and the first author, \cite{DDV}, Conjecture 5.2.

 \q The results of this paper are also used in our recent work, \cite{DG4}.   There we study the invariants of  Specht modules for a symmetric group under the action of a smaller symmetric group.    At a certain  point (in the proof of \cite{DG4},  Lemma 2.1) we use some of the theory developed here to analyse these invariants and  give a counterexample to a Conjecture of D. Hemmer,  from  \cite{H}, in each characteristic.

\q The layout of the paper is the following. Section one is preliminary and we use it to establish notation for the standard combinatorics and polynomial representation theory and connections with Hecke algebras  that we shall need. In Section 2 we describe our approach to composition factors of tensor products of symmetric powers of $E$, via the truncated symmetric powers. In Section 3 we  deal with  a reciprocity principal for decomposition numbers.  This section also contains some technical results on removal of a row or a node from   a partition  such that the corresponding simple modules occurs as a composition factor of $\barS(E)^{\otimes m}$. These principles are used repeatedly in our determination of the composition factors.

\q In Section 4 we determine for which restricted partitions the corresponding irreducible module occurs as a  composition factor of $\barS(E)^{\otimes m}$,  via the Mullineux involution on  regular partitions.  In Section 5 we introduce  distinguished partitions and describe some of their properties. In  Section 6 we complete the determination of the composition factors of $\barS(E)^{\otimes m}$.

\q Our main interest is in the classical case, but since the quantised version is essentially no more difficult  we express our results in the general context throughout.

\bigskip\bigskip
\bigskip\bigskip



\section{Preliminaries}

\subsection{Combinatorics}

\q The  standard  reference for  the polynomial representation theory of \\
$\GL_n(K)$ is the monograph \cite{EGS}.   Though we work in the quantised context this reference is appropriate  as the combinatorics is  essentially the same and we adopt the notation of \cite{EGS} wherever  convenient.  Further details may also be found in the monograph, \cite{D3}, which treats the quantised case.

\quad We begin by introducing some of the associated combinatorics.  By a partition we mean an infinite  sequence $\lambda=(\lambda_1,\lambda_2,\ldots)$ of nonnegative integers with $\lambda_1\geq\lambda_2\geq \ldots$ and $\lambda_j=0$ for $j$ sufficiently large.   If $m$ is a positive integer such that $\lambda_j=0$ for $j>m$ we identify $\lambda$ with the finite sequence $(\lambda_1,\ldots,\lambda_m)$.  The length $\l(\lambda)$ of  a partition $\lambda=(\lambda_1,\lambda_2,\ldots)$ is $0$ if $\lambda=0$ and  is the positive integer $m$ such that  $\lambda_m\neq 0$, $\lambda_{m+1}=0$, if $\lambda\neq 0$. For a partition $\lambda$, we denote  by $\lambda'$ the transpose partition of $\lambda$. We write $\P$ for the set of partitions. Let $\lambda\in \P$. We define the   degree of $\lambda=(\lambda_1,\lambda_2,\ldots)$ by $\deg(\lambda)=\lambda_1+\lambda_2+\cdots$.

\q We fix a positive integer $n$. We set $X(n)=\zed^n$. There is a natural partial order on $X(n)$. For $\lambda=(\lambda_1,\ldots,\lambda_n), \mu=(\mu_1,\ldots,\mu_n)\in X(n)$,  we write $\lambda\leq \mu$ if $\lambda_1+\cdots+\lambda_i\leq \mu_1+\cdots+\mu_i$ for $i=1,2,\ldots,n-1$ and $\lambda_1+\cdots+\lambda_n=\mu_1+\cdots+\mu_n$. We shall use the standard $\zed$-basis   $\ep_1,\ldots,\ep_n$ of   $X(n)$,  where $\ep_i=(0,\ldots,1,\ldots,0)$ (with $1$ in the $i$th position), for $1\leq i\leq n$. We write 
$\omega_i$ for the element $\ep_1+\cdots+\ep_i$ of $X(n)$, for $1\leq i\leq n$. We  denote  the element $\omega_n=(1,\dots,1)$ simply by $\omega$. We write $\Lambda(n)$ for the set of $n$-tuples of nonnegative integers. 

\q We write $X^+(n)$ for the set of dominant $n$-tuples of integers, i.e., the set of elements $\lambda=(\lambda_1,\ldots,\lambda_n)\in X(n)$ such that $\lambda_1\geq \cdots\geq  \lambda_n$. 
 We write  $\Lambda^+(n)$ for the set of partitions into at most $n$-parts, i.e.,  $\Lambda^+(n)=X^+(n)\bigcap \Lambda(n)$. We shall sometimes refer to elements of $\Lambda(n)$ as polynomial weights and to elements of $\Lambda^+(n)$ as polynomial dominant weights. For a nonnegative integer $r$ we write $\Lambda^+(n,r)$ for the set of partitions of $r$ into at most $n$ parts, i.e., the set of elements of $\Lambda^+(n)$ of degree $r$.

\q We write ${\rm Sym}(r)$ for the symmetric group on $\{1,2,\ldots,r\}$.  The symmetric group $W={\rm Sym}(n)$ acts naturally on $X(n)$.  We write $w_0$ for the longest element of  $W$,  i.e., the element such that $w_0\lambda=(\lambda_n,\ldots,\lambda_1)$, for $\lambda=(\lambda_1,\ldots,\lambda_n)\in X(n)$. 
 

\q We fix a positive integer $l$.  A partition $\lambda=(\lambda_1,\lambda_2,\ldots)$ is $l$-regular if there is no positive integer $i$ such that $\lambda_i=\lambda_{i+1}=\cdots=\lambda_{i+l-1}>0$. We write $\P_\reg$ for the set of $l$-regular partitions  and $\P_\reg(r)$ for the set of $l$-regular partitions  of degree $r$.

 \q   We write $X_1(n)$ for the set of $l$-restricted partitions into at most $n$ parts, i.e., the set of elements $\lambda=(\lambda_1,\ldots,\lambda_n)\in \Lambda^+(n)$ such that $0\leq \lambda_1-\lambda_2,\ldots,\lambda_{n-1}-\lambda_n, \lambda_n<l$.  Note that an element $\lambda\in \Lambda^+(n)$ belongs to $X_1(n)$ if and only if $\lambda'$ is an  $l$-regular partition.
 
 \q A dominant  weight $\lambda\in X^+(n)$ has a unique expression $\lambda=\lambda^0+l\barlambda$ with $\lambda^0\in X_1(n)$, $\barlambda\in X^+(n)$, moreover if $\lambda\in\Lambda^+(n)$ then $\barlambda\in \Lambda^+(n)$. We shall use this notation a great deal in what follows.

\subsection{Rational Modules and Polynomial Modules}

\q Let $K$ be a field. If $V,W$ are vector spaces over $K$, we write $V\otimes W$ for the tensor product $V\otimes_K W$.   We shall be working with the representation theory of quantum groups over $K$. By the category of quantum groups over $K$ we understand the opposite category of the category of Hopf algebras over $K$. Less formally we shall use the expression \lq\lq $G$ is a quantum group" to indicate that we have in mind a Hopf algebra over $K$ which we denote $K[G]$ and call the coordinate algebra of $G$.  We say that $\phi:G\to H$ is a morphism of quantum groups over $K$ to indicate that we have in mind a morphism of Hopf algebras over $K$, from $K[H]$ to $K[G]$, denoted $\phi^\sharp$ and called the co-morphism of $\phi$.   We will say $H$ is a quantum subgroup of the quantum group $G$, over $K$, to indicate that $H$ is a quantum group with coordinate algebra $K[H]=K[G]/I$, for some Hopf ideal $I$ of $K[G]$, which we call the defining ideal of $H$.  The inclusion morphism $i:H\to G$ is the morphism of quantum groups whose co-morphism $i^\sharp:K[G]\to K[H]=K[G]/I$ is the natural map. 

\q Let $G$ be a quantum group over $K$. The category of  left (\resp. right) $G$-modules is the the category of right (\resp. left) $K[G]$-comodules.  We write $\Mod(G)$ for the category of left $G$-modules and $\mod(G)$ for the category of finite dimensional left $G$-modules.  We shall also call a $G$-module a rational $G$-module (by analogy with the representation theory of algebraic groups).  A $G$-module will mean a left $G$-module unless  indicated otherwise. For a finite dimensional $G$-module $V$ the dual space $V^*=\Hom_K(V,K)$ has a natural $G$-module structure.
For a finite dimensional $G$-module $V$ and a non-negative integer $r$ we write $V^{\otimes r}$ for the $r$-fold tensor product $V\otimes V\otimes \cdots \otimes V$ and write $V^{\otimes -r}$ for the dual of $V^{\otimes r}$.

 \q Let $V$ be a finite dimensional $G$-module with structure map $\tau: V\to V\otimes K[G]$.  The coefficient space $\cf(V)$ of $V$ is the subspace of $K[G]$ spanned by the \lq\lq coefficient elements" $f_{ij}$, $1\leq i,j\leq m$, defined with respect to a basis $v_1,\ldots, v_m$ of $V$, by the equations 
 $$\tau(v_i)=\sum_{j=1}^m v_j\otimes f_{ji}$$
 for $1\leq i\leq m$.  The coefficient space $\cf(V)$ is independent of the choice of basis and is a subcoalgebra of $K[G]$.

\q We fix a positive integer $n$. We shall be working with $G(n)$, the quantum general linear group of degree $n$, as in \cite{D3}.   We fix a non-zero element $q$ of $K$. 
We have a $K$-bialgebra $A(n)$ given by generators $c_{ij}$, $1\leq i,j\leq n$, subject to certain relations (depending on $q$) , as in \cite{D3}, 0.20.  The comultiplication map  $\delta:A(n)\to A(n)\otimes A(n)$ satisfies $\de(c_{ij})=\sum_{r=1}^n c_{ir}\otimes c_{rj}$ and the augmentation map $\ep:A(n)\to K$ satisfies $\ep(c_{ij})=\de_{ij}$ (the Kronecker delta), for $1\leq i,j\leq n$.  The elements $c_{ij}$ will be called the coordinate elements and we define the determinant element
$$d_n=\sum_{\pi\in \Symm(n)} \sgn(\pi) c_{1,\pi(1)}\ldots c_{n,\pi(n)}.$$
Here  $\sgn(\pi)$ denotes the sign of the permutation $\pi$. We form the Ore localisation $A(n)_{d_n}$. The comultiplication map $A(n)\to A(n)\otimes A(n)$ and augmentation map $A(n)\to K$ extend uniquely to $K$-algebraic maps $A(n)_{d_n}\to A(n)_{d_n}\otimes A(n)_{d_n}$ and $A(n)_{d_n}\to K$, giving $A(n)_{d_n}$ the structure of a Hopf algebra. By the quantum general linear group $G(n)$ we mean the quantum group over $K$ with coordinate algebra $K[G(n)]=A(n)_{d_n}$.

\q We write $T(n)$ for the quantum subgroup of $G(n)$ with defining ideal generated by all $c_{ij}$ with $1\leq i,j\leq n$, $i\neq j$.  We write $B(n)$ for quantum subgroup of $G(n)$ with defining ideal generated by all $c_{ij}$ with $1\leq i <j\leq n$. We call $T(n)$ a maximal torus and $B(n)$ a Borel subgroup of $G(n)$ (by analogy with the classical case).

\q We now assign to a finite dimension rational $T(n)$-module its formal character. We form the integral group ring $\zed X(n)$. This has $\zed$-basis of formal exponentials $e^\lambda$, which multiply according to the rule $e^\lambda e^\mu=e^{\lambda+\mu}$, $\lambda,\mu\in X(n)$. For $1\leq i\leq n$ we define ${\bar c}_{ii}=c_{ii}+I_{T(n)}\in K[T(n)]$, where $I_{T(n)}$ is the defining ideal of the quantum subgroup $T(n)$ of $G(n)$.  Note that ${\bar c}_{11}\ldots {\bar c}_{nn}=d_n+I_{T(n)}$, in particular each ${\bar c}_{ii}$ is invertible in $K[T(n)]$. For $\lambda=(\lambda_1,\ldots,\lambda_n)\in X(n)$ we define ${\bar c}^\lambda={\bar c}_{11}^{\lambda_1}\ldots {\bar c}_{nn}^{\lambda_n}$. The elements ${\bar c}^\lambda$, $\lambda\in X(n)$, are group-like and form a $K$-basis of $K[T(n)]$. 
For $\lambda=(\lambda_1,\ldots,\lambda_n)\in X(n)$, we write $K_\lambda$ for $K$ regarded as a (one dimensional) $T(n)$-module with structure map $\tau:K_\lambda\to K_\lambda\otimes K[T(n)]$ given by $\tau(v)=v\otimes {\bar c}^\lambda$, $v\in K_\lambda$.  For a finite dimensional  rational $T(n)$-module $V$ with structure map $\tau:V\to V\otimes K[T(n)]$  and $\lambda\in X(n)$ we have the weight space 
$$V^\lambda=\{v\in V\vert \tau(v) =v\otimes {\bar c}^\lambda\}.$$  
Moreover, we have the weight space decomposition $V=\bigoplus_{\lambda\in X(n)} V^\lambda$. 
We say that $\lambda\in X(n)$ is a weight of $V$ if $V^\lambda\neq 0$. 
The dimension of a finite dimensional vector space $V$ over $K$ will be denoted by $\dim V$. 
The character $\ch V$ of a finite dimensional rational $T(n)$-module $V$ is the element of $\zed X(n)$ defined by
$\ch V=\sum_{\lambda\in X(n)} \dim V^\lambda e^\lambda$.   

\q For each $\lambda\in X^+(n)$ there is an irreducible rational $G(n)$-module $L_n(\lambda)$ which has unique highest weight $\lambda$ and such $\lambda$ occurs as a weight with multiplicity one. The modules $L_n(\lambda)$, $\lambda\in X^+(n)$, form a complete set of pairwise non-isomorphic irreducible rational  $G$-modules. 
 Note that for $\lambda=(\lambda_1,\ldots,\lambda_n)\in X^+(n)$ the dual module $L_n(\lambda)^*$ is isomorphic to $L_n(\lambda^*)$, where $\lambda^*=(-\lambda_n,\ldots,-\lambda_1)$. For a finite dimensional rational $G(n)$-module $V$ and $\lambda\in X^+(n)$ we write $[V:L_n(\lambda)]$ for the multiplicity of $L_n(\lambda)$ as a composition factor of $V$. 

\q We write $D_n$ for the one dimensional $G(n)$-module corresponding to the determinant. Thus $D_n$ has structure map $\tau:D_n\to D_n\otimes K[G]$, given by $\tau(v)=v\otimes d_n$, for $v\in D_n$. 
Thus we have  $D_n=L_n(\omega)=L_n(1,1,\ldots,1)$. We write $E_n$ for the natural $G(n)$-module. Thus  $E_n$ has basis $e_1,\ldots,e_n$,  and the structure map $\tau:E_n\to   E_n\otimes K[G(n)]$ is given by $\tau(e_i)=\sum_{j=1}^n e_j\otimes c_{ji}$. We also have that $E_n=L_n(1,0,\dots,0)$.

\q A finite dimensional $G(n)$-module $V$ is called polynomial if $\cf(V) \leq A(n)$. The modules $L_n(\lambda)$, $\lambda\in \Lambda^+(n)$, form a complete set of pairwise non-isomorphic irreducible polynomial $G(n)$-modules. We write $I_n(\lambda)$ for the injective envelope of $L_n(\lambda)$ in the category of polynomial modules.  We have a grading   $A(n)=\bigoplus_{r=0}^\infty  A(n,r)$ in  such a way that each $c_{ij}$ has degree $1$. Moreover each $A(n,r)$ is a finite dimensional subcoalgebra of $A(n)$. The dual algebra $S(n,r)$ is known as the Schur algebra.  A finite dimensional $G(n)$-module $V$ is polynomial of degree $r$ if $\cf(V)\leq A(n,r)$.  We write $\pol(n)$ (\resp. $\pol(n,r)$)  for the full subcategory of $\mod(G(n))$ whose objects are the polynomial modules (\resp. the modules which are polynomial of degree $r$).

\q For an arbitrary finite dimensional polynomial $G(n)$-module we may write $V$ uniquely as a direct sum $V=\bigoplus_{r=0}^\infty V(r)$ in such a way that $V(r)$ is polynomial of degree $r$, for $r\geq 0$.  Let $r\geq 0$. The modules $L_n(\lambda)$, $\lambda\in\Lambda^+(n,r)$, form a complete set of pairwise non-isomorphic irreducible polynomial $G(n)$-modules which are polynomial of degree $r$.  We write $\mod(S)$ for the category of left modules for a finite dimensional $K$-algebra $S$.  The category 
 $\pol(n,r)$ is naturally equivalent to the category $\mod(S(n,r))$.    It follows in particular  that, for $\lambda\in \Lambda^+(n,r)$, the module $I_n(\lambda)$ is a finite dimensional module which is polynomial of degree $r$.

\q We shall also need modules induced from $B(n)$ to $G(n)$.  (For details of the induction functor $\Mod(B(n))\to \Mod(G(n))$ see, for example, \cite{DStd}.)  For $\lambda\in X(n)$ there is a unique (up to isomorphism) one dimensional $B(n)$-module whose restriction to $T(n)$ is  $K_\lambda$. We also denote this module by $K_\lambda$. The induced module $\ind_{B(n)}^{G(n)} K_\lambda$ is non-zero if and only if $\lambda\in X^+(n)$. For $\lambda\in X^+(n)$ we set $\nabla_n(\lambda)=\ind_{B(n)}^{G(n)}  K_\lambda$. Then $\nabla_n(\lambda)$ is finite dimensional  and its character is the Schur symmetric function corresponding to $\lambda$. The $G(n)$-module socle of $\nabla_n(\lambda)$ is $L_n(\lambda)$. The module $\nabla_n(\lambda)$ has unique highest weight $\lambda$ and this weight occurs with multiplicity one.    For $\lambda\in X^+(n)$ we take as a definition of the Weyl module $\Delta_n(\lambda)$ the dual module $\nabla_n(-w_0\lambda)^*$. Thus $\nabla_n(\lambda)$ and $\Delta_n(\lambda)$ have the same character.

\q A filtration $0=V_0\leq V_1\leq \cdots\leq V_r=V$ of  a finite dimensional rational $G(n)$-module $V$ is said to be {\it good} if for each $1\leq i\leq r$ the quotient $V_i/V_{i-1}$ is either zero or isomorphic to $\nabla_n(\lambda^i)$ for some $\lambda^i\in X^+(n)$.  For a rational $G(n)$-module $V$ admitting a good filtration for each $\lambda\in X^+(n)$, the multiplicity 
$|\{1\leq i\leq r\vert V_i/V_{i-1}\cong \nabla_n(\lambda)\}|$ is independent of the choice of the good filtration, and will be denoted $(V:\nabla_n(\lambda))$. 

\q For $\lambda,\mu\in X^+(n)$ we have $\Ext^1_{G(n)}(\nabla_n(\lambda),\nabla_n(\mu))=0$ unless $\lambda>\mu$. Given Kempf's Vanishing Theorem, \cite{D3},  Theorem 3.4,  this follows exactly as in the classical case, e.g., \cite{D4},  Lemma 3.2.1 (or the original source \cite{CPS}, Corollary (3.2)). It  follows that if $V$ has a good  filtration 
$0=V_0\leq V_1\leq \cdots\leq V_t =V$ with sections $V_i/V_{i-1}\cong \nabla_n(\lambda_i)$, $1\leq i\leq t$, and $\mu_1,\ldots,\mu_t$ is a reordering of the $\lambda_1,\ldots,\lambda_t$ such that $\mu_i<\mu_j$ implies that $i<j$ then there is a good filtration $0=V_0'<V_1'<\cdots <V_t'=V$ with $V_i'/V_{i-1}'\cong \nabla_n(\mu_i)$, for $1\leq i\leq t$. 

\q Similarly it will be of great practical use  to know that\\
 $\Ext^1_{G(n)}(\nabla_n(\lambda),\nabla_n(\mu))=0$ when $\lambda$ and $\mu$ belong to different blocks. Here  the relationship with cores of partitions diagrams (discussed later) will be crucial for us.    For a partition $\lambda$ we denote by $[\lambda]$ the corresponding partition diagram (as in \cite{EGS}). The $l$-core of $[\lambda]$ is the diagram obtained by removing skew $l$-hooks,  as in \cite{James}. If $\lambda,\mu\in \Lambda^+(n,r)$ and $[\lambda]$ and $[\mu]$ have different $l$-cores then the simple modules $L_n(\lambda)$ and $L_n(\mu)$ belong to different blocks and it follows in particular that $\Ext^i_{S(n,r)}(\nabla(\lambda),\nabla(\mu))=0$, for all $i\geq 0$.  A precise  description of the blocks of the $q$-Schur algebras was found by Cox, see \cite{Cox},  Theorem 5.3.

\q  For $\lambda\in \Lambda^+(n)$  the module $I_n(\lambda)$ has a good filtration and we have the reciprocity formula $(I_n(\lambda):\nabla_n(\mu))=[\nabla_n(\mu):L_n(\lambda)]$ see e.g., \cite{DStd}, Section 4, (6).

\subsection{The Frobenius Morphism}

\q It will be important for us to make a comparison with the classical case $q=1$.   In this case we will write $\dotG(n)$ for  $G(n)$ and write $x_{ij}$ for the coordinate element $c_{ij}$, $1\leq i,j\leq n$.   In this case we write $\dotL_n(\lambda)$ for the $\dotG$-module $L_n(\lambda)$, $\lambda\in X^+(n)$, and write $\dotE_n$ for $E_n$.

\q  We return to the general situation.  If $q$ is not a root or unity, or if $K$ has characteristic $0$ and $q=1$ then all $G(n)$-modules are completely reducible, see e.g.,  \cite{DStd}, Section 4, (8).  We therefore assume from now on that $q$ is a root of unity and that if $K$ has characteristic $0$ then $q\neq 1$.  Also, from now on, $l$ is the smallest positive integer such that $1+q+\cdots+q^{l-1}=0$.

\q Now we have a morphism of Hopf algebras $\theta: K[\dotG(n)]\to K[G(n)]$ given by $\theta(x_{ij})=c_{ij}^l$, for $1\leq i,j\leq n$.  We write $F:G(n)\to \dotG(n)$ for the morphism of quantum groups such that $F^\sharp=\theta$.  Given a $\dotG(n)$-module $V$ we write $V^F$ for the corresponding $G(n)$-module. Thus, $V^F$ as a vector space is $V$ and if the $\dotG(n)$-module $V$ has structure map $\tau:V\to V\otimes K[\dotG(n)]$ then $V^F$ has structure map $(\id_V \otimes F)\circ \tau: V^F \to  V^F\otimes K[G(n)]$, where $\id_V:V\to V$ is the identity map on the vector space  $V$.  

\q  For an element $\phi=\sum_{\xi\in X(n)} a_\xi e^\xi$ of $\zed X(n)$ we write $\phi^F$ for the element $\sum_{\xi\in X(n)}a_\xi e^{l\xi}$.  Then, for a finite dimensional $\dotG(n)$-module $V$ we have $\ch V^F= (\ch V)^F$.   Moreover, we have the following relationship between the irreducible modules for $G(n)$ and ${\dotG(n)}$, see \cite{D3}, Section 3.2, (5).

\bs

{\bf 1.3.1  Steinberg's Tensor Product Theorem}  \sl For  $\lambda^0\in X_1(n)$ and $\barlambda\in X^+(n)$ we have 
$$L_n(\lambda^0+l\barlambda)\cong L_n(\lambda^0)\otimes \dotL_n(\barlambda)^F.$$

\rm

\bs

\q  Usually we shall abbreviate the quantum groups  $G(n)$, $B(n)$, $T(n)$ to  $G$, $B$, $T$ and $\dotG(n)$ to $\dotG$.   Likewise,  we usually abbreviate the modules 
$L_n(\lambda)$, $\nabla_n(\lambda)$, $\Delta_n(\lambda)$,  $I_n(\lambda)$ and $\dotL_n(\lambda)$ to $L(\lambda)$, $\nabla(\lambda)$, $\Delta(\lambda)$,  $I(\lambda)$ and $\dotL(\lambda)$, for $\lambda\in \Lambda^+(n)$, and abbreviate the modules $E_n$ and $D_n$ to $E$ and $D$.

\subsection{A truncation functor}

\q Let $N,n$ be positive integers with $N\geq n$.   We identify $G(n)$ with the quantum subgroup of $G(N)$ whose defining ideal is generated by all $c_{ii}-1$, $n<i\leq N$, and all $c_{ij}$ with $1\leq i\neq j\leq N$ and $i>n$ or $j>n$.  We have an exact functor (the truncation functor) $d_{N,n}:\pol(N)\to \pol(n)$ taking $V\in \pol(N)$ to the $G(n)$ submodule  $\bigoplus_{\alpha\in \Lambda(n)}  V^\alpha$  of $V$ and taking a morphism of polynomial modules $V\to V'$ to its restriction $d_{N,n}(V)\to d_{N,n}(V')$. For a discussion of this functor at the level of modules for Schur algebras in the classical case see \cite{EGS}, Section 6.5.

\q   For a finite sequence of nonnegative integers $\alpha=(\alpha_1,\ldots,\alpha_m)$ we write $S^\alpha (E_n)$ for the tensor product of symmetric powers $S^{\alpha_1}(E_n)\otimes \cdots\otimes  S^{\alpha_m}(E_n)$.

\bs
 
{\bf Proposition 1.4.1 } \sl The functor $d_{N,n}$ has the following properties:

(i) for polynomial $G(N)$-modules $X,Y$ we have $d_{N,n}(X\otimes Y)=d_{N,n}(X)\otimes d_{N,n}(Y)$;\\
 (ii) for $\alpha$ a finite sequence of nonnegative integers we have $d_{N,n} S^\alpha(E_N)=S^\alpha(E_n)$;\\
(iii)  for $\lambda\in \Lambda^+(N,r)$ and $X_\lambda=L_N(\lambda), \nabla_N(\lambda)$ or $\Delta_N(\lambda)$ then $d_{N,n}(X_\lambda)\neq0$ if and only if $\lambda\in \Lambda^+(n,r)$;\\
(iv)  for  $\lambda\in \Lambda^+(n,r)$, $d_{N,n}(L_N(\lambda))=L_n(\lambda)$, $d_{N,n}(\nabla_N(\lambda))=\nabla_n(\lambda)$ and $d_{N,n}(\Delta_N(\lambda))=\Delta_n(\lambda)$.

\rm

\begin{proof}  Part (i) is immediate. Part (ii) is an easy check as is part (iii). Part (iv) follows from \cite{D3}, 4.2, (4). 

\end{proof}

\subsection{Connections with the Hecke algebras}

\q We now record some connections with representations of  Hecke algebra of type $A$. We fix a positive integer $r$.  We write $\l(\pi)$ for the length of a permutation $\pi$.   The Hecke algebra $\Hec(r)$ is the $K$-algebra with basis $T_w$, $w\in \Symm(r)$, and multiplication satisfying

\begin{align*}&T_wT_{w'}=T_{ww'}, \quad \hbox{ if } \l(ww')=\l(w)+\l(w'), \hbox{and}\cr
&(T_s+1)(T_s-q)=0
\end{align*}
for $w,w'\in \Symm(r)$ and a basic transposition $s\in \Symm(r)$.

\q Assume now $n\geq r$.  We have the Schur functor $f:\mod(S(n,r))\to \mod(\Hec(r))$, see \cite{D3}, 2.1.  For $\lambda$ a partition of degree $r$ we denote  by $\Sp(\lambda)$ the corresponding (Dipper-James) Specht module.

\bs

{\bf Proposition 1.5.1}  \sl The functor $f$ has the following properties :\\
(i) $f$ is exact;\\
(ii) for $\lambda\in \Lambda^+(n,r)$ we have $f\nabla_n(\lambda)=\Sp(\lambda)$;\\
(iii)  for  $\lambda\in \Lambda^+(n,r)$ we have  $f(L_n(\lambda))\neq 0$ if and only if $\lambda\in X_1(n)$ and the set $\{f(L_n(\lambda))| \lambda\in X_1(n)\}$ is a full set of  pairwise non-isomorphic simple $\Hec(r)$-modules.

\rm
\bs

\begin{proof} (i) is clear from the definition. For (ii) see \cite{D3} Proposition 4.5.8.  and for (iii) see \cite{D3}, 4.3, (9) and 4.4,(2).

\end{proof}

\q There is an alternative description of the irreducible $\Hec(r)$-modules. For $\lambda\in \P_\reg(r)$,  we define $D^\lambda$ (denoted $D(\lambda)$ in \cite{D3}) to be the head of the Specht module $\Sp(\lambda)$.  Then $D^\lambda$, $\lambda\in \P_\reg(r)$, is a complete set of pairwise non-isomorphic simple $\Hec(r)$-modules.  The relationship between these two labelings of the  irreducible modules will be crucial for us in what follows.

\q We use the notation of \cite{D3}, Section 4.4.  There is an involutory algebra automorphism $\sharp: \Hec(r)\to \Hec(r)$ given by $\sharp(T_s)=-T_s+(q-1)1$, for a basic  transposition $s\in \Symm(r)$.  For a $\Hec(r)$-module $V$ affording the representation $\pi:\Hec(r)\to \End_K(V)$ we write $V^\sharp$ for the $K$-space $V$ regarded as a module via the representation $\pi\circ\sharp$.

\q   The relationship between the labellings is:

$$fL(\lambda)\cong (D^{\lambda'})^\sharp$$

for $\lambda\in X_1(n)$.

\q Therefore a direct relation between  the two descriptions of the irreducible modules for the symmetric group is  described in terms of the involution   $\P_\reg(r)\to \P_\reg(r)$, $\lambda\mapsto \tilde\lambda$  defined by 
$(D^\lambda)^\sharp \cong D^{\tilde\lambda}$. This bijection is named after G. Mullineux,  who proposed, in \cite{Mull},  an algorithm to describe it explicitly in the classical case $q=1$ and $K$ a field of  characteristic $p$.  The algorithm proposed  by Mullineux makes perfect sense also in the quantised case.     We write $\Mull:\P_\reg(r)\to \P_\reg(r)$ for this bijection and call it the Mullineux involution. Thus we have
$$f (L_n(\lambda))\iso D^{\Mull(\lambda')}$$
for $\lambda$ an  $l$-restricted partition of degree $r$.

\q Mullineux's original conjecture was proved by Ford  and Kleshchev in \cite{FK}.  The quantised version was proved by Brundan, \cite{JB}. 
This bijection is very important to us and we shall assume some familiarity with the Mullineux algorithm in later sections.

\q We state explicitly some of the most important properties of this map for us. We indicate an argument here since it will be important for us.  The argument is essentially in \cite{D0} (in the classical case)  but it is perhaps more convenient to use the language of tilting modules, as in \cite{D3}.   For $\lambda\in \Lambda^+(n,r)$ we write $T_n(\lambda)$ for the corresponding tilting module, as in \cite{D3}.

\bs 

{\bf Proposition 1.5.2}  \sl Let $\lambda$ be a restricted partition of $r$ and let $\mu=\Mull(\lambda')$.  Then   $\mu$ is the unique maximal element of the set,\\
  $S=\{\tau\in \Lambda^+(n,r) \vert [\nabla_n(\tau):L_n(\lambda)]\neq 0\}$.
\bs
\rm

\begin{proof} Let  $\tau\in \Lambda^+(n,r)$. We have $[\nabla_n(\tau):L_n(\lambda)]= (I_n(\lambda):\nabla_n(\tau))$. Moreover, we have $I_n(\lambda)=T_n(\Mull(\lambda'))$, \cite{D3}, 4.3, (10),  so that $\tau\in S$ if and only if $(T_n(\Mull(\lambda')):\nabla_n(\mu)) \neq 0$. But $T_n(\Mull(\lambda'))$ has unique highest weight $\Mull(\lambda')$ so the result follows.

\end{proof}


\bigskip\bigskip
\bigskip\bigskip


\section{Special Partitions and Good Partitions}

\q The symmetric algebra $S(E_n)$ has  the homogeneous  ideal and $G(n)$-submodule $I$ generated by $e_1^l,\ldots,e_n^l$. We write $\barS(E_n)$ for the quotient $S(E_n)/I$.  Then $\barS(E_n)$ inherits a grading and $G(n)$-module decomposition $\barS(E_n)=\bigoplus_{r=0}^\infty \barS^r(E_n)$.   The images of the elements $e_1^{r_1}\ldots e_n^{r_n}$, with $0\leq r_1,\ldots,r_n\leq l-1$, $r_1+\cdots+r_n=r$  form a basis of $\barS^r(E_n)$, for $r\geq 0$.

\q Let $m\leq n$. For $\alpha=(\alpha_1,\ldots,\alpha_m)\in \Lambda(m)$ we define $\barS^\alpha(E_n)=\barS^{\alpha_1}(E_n)\otimes \cdots \otimes S^{\alpha_m}(E_n)$. Thus we have $\barS(E_n)^{\otimes m}=\bigoplus_{\alpha\in \Lambda(m)} \barS^\alpha(E_n)$. 

\q We are now ready to make two key definitions.

\begin{definition} Let $m\geq 1$. 

(i) We will say that  $\lambda\in \Lambda^+(n)$ is  $m$-good (with respect to $n$)   if   $L_n(\lambda)$ is a composition factor of the $m$-fold tensor product  $S(E_n)^{\otimes m}$.

(ii) We will say that  $\lambda\in \Lambda^+(n)$ is  $m$-special (with respect to $n$)  if   $L_n(\lambda)$ is a composition factor of $\barS(E_n)^{\otimes m}$. 
\end{definition}

From  \cite{DG2}, Lemma 3.8  we get:

\begin{lemma} An element $\lambda\in \Lambda^+(n)$ is $m$-good if and only if there exists $\mu\in \Lambda^+(n)$ of  length at most $m$ such that $[\nabla_n(\mu):L_n(\lambda)]\neq 0$.
\end{lemma}

\begin{lemma} (The Stability Properties)  Let $m,n,N $  be positive integers with $N\geq n$.  Let $\lambda$ be a partition of length at most $n$. Then $\lambda$ is $m$-good (\resp. $m$-special) with respect to $n$ if and only if $\lambda$ is $m$-good (\resp. $m$-special) with respect to $N$.

\end{lemma}

\begin{proof}  For  $\alpha\in   \Lambda(m)$ we have   $d_{N,n}(S^\alpha(E_N))=S^\alpha(E_n)$  from  Proposition 1.4.1 (ii) and this, together with Proposition 1.4.1 (iv) 
  gives the result for $m$-good partitions. 

\q  It is easy to check, from the explicit   bases of $\barS(E_N)$  and $\barS(E_n)$, as above,  that $d_{N,n}(\barS(E_N))=\barS(E_n)$ from which we get  that $d_{N,n}(\barS(E_N)^{\otimes m})=\barS(E_n)^{\otimes m}$ by Proposition 1.4.1 (i).  Now Proposition 1.4.1 (iv) gives the result for $m$-special partitions.
\end{proof}

\bf Notation\q\rm In view of the above lemma, for a positive integer $m$,  we shall say that a partition $\lambda$ is $m$-good (\resp.  $m$-special) if it is $m$-good (\resp.  $m$-special) with respect to $n$, for $n\geq \l(\lambda)$.

\q We record an elementary observation.

\begin{lemma} Let $\lambda,\mu\in \Lambda^+(n)$ and let $m_1,m_2\geq 0$.

If  $\lambda$ is $m_1$-good (resp.  $m_1$-special)  and $\mu$ is $m_2$-good (resp. $m_2$-special) then  $\lambda+\mu$ is $(m_1+m_2)$-good (resp. $(m_1+m_2)$-special).

\end{lemma}

\begin{proof}    We suppose $n$ is sufficiently large. Let $S=S(E_n)$ and suppose that $\lambda\in \Lambda^+(n)$ is $m_1$-good and $\mu\in \Lambda^+(n)$ is $m_2$-good. Then $L_n(\lambda)$ occurs as a section of $S^{\otimes m_1}$ and $L_n(\mu)$ occurs as a section of  $S^{\otimes m_2}$. Hence $L_n(\lambda)\otimes L_n(\mu)$ occurs as a section of $S^{\otimes (m_1+m_2)}=S^{\otimes m_1}\otimes S^{\otimes m_2}$.  Now $L_n(\lambda)\otimes L_n(\mu)$ has highest weight $\lambda+\mu$ so that $L_n(\lambda+\mu)$ occurs as a composition factor of $L_n(\lambda)\otimes L_n(\mu)$, and hence of $S^{\otimes (m_1+m_2)}$, i.e., $\lambda+\mu$ is $(m_1+m_2)$-good.

\q The argument for special partitions is completely analogous.

\end{proof}

\q We elucidate the relationship between $m$-good and $m$-special partitions via some properties   of graded modules that we now recall.
Let $A$ be a $K$-algebra. If $M$ is a left $A$-module, $S$ is a subspace of $A$ and $V$ is a subspace of $M$ then we write $SV$ for the subspace of $M$ spanned by all elements $sv$, with $s\in S$, $v\in V$. Now suppose that $A$ has a $K$-algebra grading 
 $A=\bigoplus_{r=0}^\infty A_r$. We assume further  that $A_0=K$ and that $A_1$ generates $A$ and has finite dimension. We set $A_+=\sum_{r>0}A_r$. 

\q Let $M=\bigoplus_{i\geq0}M_i$ be a finitely generated graded $A$-module and consider the graded vector space $\barM=M/A_+M$

\begin{lemma}
If $V$ is a homogeneous  subspace of $M$ such that 
$$M_r=(A_+M)_r+V_r$$
for all $r$  then the multiplication map  $A\otimes V\rightarrow M$ is  surjective.

\end{lemma}

\begin{proof}

We have  $M_0=V_0\leq AV$. Now assume $r>0$ and  $M_j\leq AV$ for $j<r$. Then
$$M_r=(A_+M)_r+V_r\leq\sum_{j<r}AM_j+V\leq AV$$
so it  follows by induction that $M_r\leq AV$ for all $r$. Hence $AV=M$, i.e., the map $A\otimes V\to M$ is surjective.

\end{proof}

\begin{proposition}

(i) The $A$-module $M$ is graded free if and only if $$\dim M_r=\sum_{i+j=r}\dim A_i \,. \dim\barM_j$$
for all $r\geq0$

(ii) Assume that $M$ is graded free.  A  homogenous subspace $V$ is free generating space  (i.e., multiplication $A\otimes V\to M$ is an isomorphism)  if and only if the natural map $V\rightarrow \barM$ is an isomorphism.

\end{proposition}

\begin{proof}

Assume that $M$ is graded free and $V$  is a homogeneous subspace freely generating $M$. Then the multiplication map $A\otimes V\rightarrow M$ is a linear  isomorphism  and induces an isomorphism, 

$$A\otimes V/A_+\otimes V\rightarrow\barM.$$

Hence the natural map $V\rightarrow \barM$ is an isomorphism.

\q We give $A\otimes V$ a grading with $A\otimes V=\bigoplus_{r=0}^\infty (A\otimes V)_r$, with $(A\otimes V)_r=\sum_{r=i+j} A_i\otimes V_j$, for $r\geq 0$.
The  isomorphism $A\otimes V\rightarrow M$ gives
$$\dim (A\otimes V)_r=\dim M_r$$
i.e.,
$$ \sum_{r=i+j} \dim A_i . \dim V_j= \dim M_r$$
and hence 
$$ \sum_{r=i+j} \dim A_i . \dim \barM_j= \dim M_r$$
for all $r\geq 0$.

\q  Suppose conversely  that $\sum _{i+j=r}\dim A_i\dim \barM_j=\dim M_r$ for all $r$.  Let $V$ be any homogeneous subspace of $M$ such that the natural map $V\to \barM$ is an isomorphism, i.e., $V=\oplus_{r=0}^\infty V_r$, where $V_r$ is a complement of  $(A_+M)_r$ in $M_r$  for each $r$.
By  the Lemma 2.1 above, the multiplication map   $A\otimes V\rightarrow M$ is  surjective. Hence the map 
$$\bigoplus_{r=i+j}A_i\otimes V_j\rightarrow M_r$$
is onto for all $r$. But 
$$\sum_{r=i+j}\dim A_i. \dim V_j=\sum_{r=i+j} \dim A_i . \dim \barM_j=M_r.$$
 Therefore, the above map is an isomorphism and so the multiplication map is an isomorphism. Hence $M$ is freely  generated by $V$. This proves everything.

\end{proof}

\q  We now suppose that $A$ and $M$ are $T(n)$-modules in such a way that the gradings $A=\bigoplus_{r=0}^\infty A_r$  and $M=\bigoplus_{r=0}^\infty M_r$ are module homomorphisms and that  multiplication the multiplication map $A\otimes A\to A$ the action $A\otimes M\to M$ are  $T(n)$-module homomorphisms.

\begin{proposition} Assume that $M$ is graded free and let $V_r$ be a $T(n)$-module complement of $(A_+M)_r$ in $M_r$, for each $r$,  and form the  $T(n)$-module $V=\bigoplus_{r\geq0}V_r$.  Then, for $r\geq 0$,  we have 

$$M_r\cong \bigoplus_{i+j=r}A_i\otimes V_j$$ 
as $T(n)$-modules.

\end{proposition}

\bs

\q  We shall apply the above generalities to a tensor product of copies of the symmetric algebra $S(E)$ on the natural module $E$ for $G(n)$. Let $S=S(E_n)$. Then $S$ has  the subalgebra $R$ generated $e_1^l,\ldots,e_n^l$.  We note that $R$ is a $G(n)$-submodule and in fact $R$ is isomorphic to $S(\dotE_n)^F$, (where $F:G\to \dotG$ is the Frobenius morphism), via the $K$-algebra map taking $e_i\in \dotE_n$ to $e_i^l\in S(E_n)$.  

\q Let $m\geq 0$.  We set $A=R\otimes \cdots \otimes R$ ($m$ times) and $H=S\otimes \cdots\otimes S$ ($m$ times).  We regard $S$ as a module over $R$, via the inclusion map and hence $H=S\otimes \cdots\otimes S$ as a module over $A=R\otimes \cdots \otimes R$.  As a $G(n)$-module we have $A\iso \dotH^F$, where $\dotH= S(\dotE_n)\otimes \cdots \otimes S(\dotE_n)$.  The natural map $S^{\otimes m}\to \barH$ induces an isomorphism $\barS^{\otimes m}\to \barH$.

\q Suppose $M$ is a polynomial $G(n)$-module with decomposition with homogenous component $M_r$ of degree $r$, for $r\geq 0$, and each $M_r$ is finite dimensional. Then, for $\lambda\in \Lambda^+(n,r)$, we write $[M:L_n(\lambda)]$ for $[M_r:L_n(\lambda)]$.

\bs

\begin{proposition}For $\lambda\in \Lambda^+(n)$ we have

\begin{align*}[S&(E_n)^{\otimes m}:L_n(\lambda)]\cr
&=\sum_{\mu,\tau\in\Lambda^+(n)}[\barH:L_n(\mu)]\times [S(\dotE_n)^{\otimes m}:\dotL_n(\tau)]\times [L_n(\mu)\otimes \dotL_n(\tau)^F:L_n(\lambda)].
\end{align*}

\end{proposition}

\begin{proof}

Let $r$ be the degree of $\lambda$. Then we have  $[H:L_n(\lambda)]=[H_r:L_n(\lambda)]$. Now by Proposition 2.7  the $T(n)$-modules $H_r$ and $\bigoplus_{r=i+j} \barH_i\otimes A_j$ have the same character. Hence we have
\begin{align*}
[H:&L_n(\lambda)]=\sum_{r=i+j} [\barH_i\otimes A_j:L_n(\lambda)]=\sum_{r=i+lj} [\barH_i\otimes \dotH_j^F:L_n(\lambda)]\cr
&=\sum_{\substack{r=i+lj,\\ \mu,\tau\in\Lambda^+(n)}}  [\barH_i:L_n(\mu)]\times [\dotH_j^F:\dotL_n(\tau)^F]\times [L_n(\mu)\otimes \dotL_n(\tau)^F:L_n(\lambda)]\cr
&=\sum_{\mu,\tau\in\Lambda^+(n)}[\barH:L_n(\mu)]\times [\dotH:\dotL_n(\tau)]\times [L_n(\mu)\otimes \dotL_n(\tau)^F:L_n(\lambda)]
\end{align*}
as required.

\end{proof}

\q If $K$ has  characteristic $p>0$  then we have the usual Frobenius $\dotF:\dotG(n)\to \dotG(n)$, whose comorphism takes $c_{ij}$ to $c_{ij}^p$, for $1\leq i,j\leq n$.  In that case we write $J$ for $H$ and $\barJ$ for $\barH$. Repeating the above 
 Proposition  we obtain  the following.

\begin{corollary} Suppose $K$ has positive characteristic. Let $\lambda\in \Lambda^+(n)$. Then, for all sufficiently large $N$ (depending on $\lambda$)  we have:

(i) $[J:L_n(\lambda)]=[\barH\otimes  (\barJ \otimes \barJ^{\dotF} \dots\otimes \barJ^{\dotF^{N-1}})^F:L_n(\lambda)]$; and 

(ii) $\lambda$ is $m$-good if and only if  there exists an element   $\mu^0\in \Lambda^+(n)$ which is $m$-special for $G(n)$ and elements $\mu^1,\dots,\mu^N\in \Lambda^+(n)$  which are $m$-special for $\dotG(n)$  such that
$[L_n(\mu^0)\otimes (\dotL_n(\mu^1) \otimes \dotL_n(\mu^2)^\dotF \dots \otimes \dotL_n(\mu^N)^{\dotF^{N-1}})^F:L_n(\lambda)]\neq 0$.

\end{corollary}


\bigskip\bigskip
\bigskip\bigskip


\section{Reciprocity, Row Removal and Node Removal}

\q The following will be useful to us immediately and in Section 6. 

\begin{lemma} Let $\lambda$ be a non-restricted partition. If $\mu$ is a partition such that $L(\mu)$ is a composition factor of $L(\lambda)\otimes V$, for some polynomial module $V$, then $\mu$ is non-restricted.
\end{lemma}

\begin{proof} We write $\lambda=\lambda^0+l\barlambda$, with $\lambda^0,\barlambda$ partitions with  $\lambda^0$ restricted and $\barlambda\neq 0$. Then $L(\lambda)\otimes V=L(\lambda^0)\otimes \dotL(\barlambda)^F\otimes V$ so that $L(\mu)$ is a composition factor of $\dotL(\barlambda)^F\otimes L(\tau)$, for some partition $\tau$ such that $L(\tau)$ is a composition factor of $L(\lambda^0)\otimes V$. We have $\tau=\tau^0+l\bartau$, for partitions $\tau^0,\bartau$, with $\tau^0$ restricted. Then $L(\mu)$ is a composition factor of $L(\tau^0)\otimes (\dotL(\bartau)\otimes \dotL(\barlambda))^F$ and hence, by Steinberg's tensor product theorem, we have $\mu=\tau^0+l\barmu$,  where $\dotL(\barmu)$ is a composition factor of $\dotL(\bartau)\otimes \dotL(\barlambda)$. Thus $\dotL(\barmu)$ is polynomial of degree 
$$\deg(\bartau)+ \deg(\barlambda)\geq \deg(\barlambda)>0.$$
Thus $\barmu\neq 0$ and $\mu$ is not restricted.

\end{proof}

We shall also need the following result.

\begin{proposition}

Let $m$ be a positive integer. For  a restricted partition $\lambda$,   the following are equivalent:

(i) $\lambda$ is $m$-good;

(ii) $\lambda$ is $m$-special;

(iii) $\l(\Mull(\lambda'))\leq m$.

\end{proposition}

\begin{proof}  We work with modules for quantum general linear groups of degree $n\geq r=\deg(\lambda)$.

(i) $\Rightarrow$ (ii) 
Suppose that $\lambda$ is $m$-good. Then  putting $H=S(E)^{\otimes m}$ we have $[H:L(\lambda)]\neq0$. By Proposition 2.8  there exist partitions $\mu$ and $\tau$ such that $\mu$ is $m$-special and $[L(\mu)\otimes \dotL(\tau)^F:L(\lambda)]\neq0$.  By Lemma 3.1, $\tau=0$, so that $\lambda=\mu$, which is $m$-special.

\smallskip

(ii) $\Rightarrow$ (i) This is clear.

\smallskip

(i) $\Rightarrow$ (iii)  Since  $\lambda$ is $m$-good, by Lemma 2.2  we have  $[\nabla(\mu):L(\lambda)]\neq0$ for some partition $\mu$ with at most $m$ parts. Hence applying the Schur functor $f:\mod(S(n,r))\to \mod(\Hec(r))$ we get 
$$[f\nabla(\mu):f L(\lambda)]=[\Sp(\mu):D^{\m(\lambda')}]\neq0.$$  

Now,  by \cite{James}, Corollary 12.2. we get that $\m(\lambda')\geq \mu$ and $\Mull(\lambda')$ has  at most $m$ parts.

\smallskip

(iii) $\Rightarrow$ (i). Suppose that  $\m(\lambda')$ has at most $m$ parts   and write  $\mu=\m(\lambda')$.  We have that $\nabla(\mu)$ appears as a section of a good filtration of $S^\mu E$, see e.g., \cite{DG1}, Lemma 3.8.   Moreover applying the Schur functor  to  $[\nabla(\mu):L(\lambda)]$  we get that,
 $$[\nabla(\mu):L(\lambda)]=[\Sp(\mu):D^\mu]=1$$
and $\lambda$ is $m$-good by Lemma 2.2.

\end{proof}

\q We fix $n$. For $\lambda=(\lambda_1,\ldots,\lambda_n)\in \Lambda^+(n)$ with $\lambda_1\leq m(l-1)$ we define $\lambda^\dagger\in \Lambda^+(n)$ by  
$$\lambda^\dagger=(m(l-1)-\lambda_n,\ldots, m(l-1)-\lambda_2,m(l-1)-\lambda_1).$$

\begin{remark}   For a finite dimensional $G(n)$-module and $\lambda\in X^+(n)$ the composition multiplicity $[V:L(\lambda)]$ is the coefficient $a_\lambda$ of $\ch L(\lambda)$ in the expression 
$\ch V=\sum_{\mu\in X^+(N)} a_\mu \ch L(\mu)$ (with all $a_\mu$ non-negative integers). If follows that 
for  finite dimensional $G(n)$-modules $U,V$ and $\lambda\in X^+(n)$ we have $[U\otimes V:L(\lambda)]=[U^*\otimes V^*:L(\lambda^*)]$.  This observation will be used in the proof of the following result.

\end{remark}

\begin{lemma} (Reciprocity Principle.) Let $\lambda=(\lambda_1,\ldots,\lambda_n)\in \Lambda^+(n)$ with $\lambda_1\leq m(l-1)$.  Then $\lambda$ is $m$-special if and only if $\lambda^\dagger$ is $m$-special.

\end{lemma}

\begin{proof}  Let $S=S(E)$ and $\barS=\barS(E)$.   The images of the elements $e_1^{a_1}\ldots e_n^{a_n}$, with $0\leq a_1,\ldots,a_n\leq l-1$ and $a_1+\dots +a_n=r$, form a basis of $\barS_r$.  In particular we have $\barS_{n(l-1)}\cong D^{\otimes (l-1)}$ and  $\barS_j=0$ for $j>n(l-1)$.  Let $0\leq i\leq n(l-1)$. Then the multiplication map $\barS_j\otimes \barS_{n(l-1)-j}\to \barS_{n(l-1)}$ is a $G(n)$-module map and a perfect pairing of $K$-spaces. Hence we have the natural isomorphism 
$$\barS_j\to \Hom_K(\barS_{n(l-1)-j}, \barS_{n(l-1)})=\barS_{n(l-1)-j}^*\otimes D^{\otimes (l-1)}.$$

\q Now we consider $\barH\iso \barS^{\otimes m}$.  Suppose $\lambda$ has degree $r$. Then we have $[\barH:L(\lambda)]=[\barH_r:L(\lambda)]$ and 
$$\barH_r \iso \bigoplus_{r=r_1+\cdots + r_m} \barS_{r_1}\otimes \cdots \otimes \barS_{r_m}$$
so that $[\barH:L(\lambda)]\neq 0$ if and only if there exists $r_1,\ldots,r_m\geq 0$ such that $r=r_1+\cdots + r_m$ and $[\barS_{r_1}\otimes \cdots \otimes \barS_{r_m}:L(\lambda)]\neq 0$. Moreover, we have 
$$[\barS_{r_1}\otimes \cdots \otimes \barS_{r_m}:L(\lambda)]=[\barS_{t_1}^*\otimes D^{\otimes (l-1)}\otimes\cdots \otimes \barS_{t_m}^*\otimes  D^{\otimes (l-1)}:L(\lambda)]$$
where $t_i=n(l-1)-r_i$, $1\leq i\leq m$. Dualising we thus get 
$$[\barS_{r_1}\otimes\cdots \otimes    \barS_{r_m}:L(\lambda)]    =[\barS_{t_1}\otimes\cdots \otimes \barS_{t_m}\otimes  D^{\otimes -m(l-1)}: L(\lambda)^*]$$
and this is 
$$[\barS_{t_1}\otimes\cdots \otimes \barS_{t_m}:  D^{\otimes m(l-1)}\otimes L(\lambda^*)].$$
But now $\lambda^*=(-\lambda_n,\ldots,-\lambda_2,-\lambda_1)$ and so 
\begin{align*}D^{\otimes m(l-1)}\otimes L(\lambda^*)&=L(m(l-1)-\lambda_n,\ldots,m(l-1)-\lambda_2,m(l-1)-\lambda_1)\cr
&= L(\lambda^\dagger)
\end{align*}
and the result follows.

\end{proof}

\q Combining the stability and reciprocity principles we deduce the following.

\begin{proposition} Let $m\geq 1$ and let $\lambda$ be a partition with $\lambda_1=m(l-1)$. Then $\lambda$ is $m$-special if and only if $(\lambda_2,\lambda_3,\ldots)$ is $m$-special.

\end{proposition}

\begin{proof}  Suppose $\lambda$ has length $n$. Then, applying the reciprocity principle, we have that  $\lambda$ is $m$-special if and only if $(m(l-1)-\lambda_n,m(l-1)-\lambda_{n-1},\ldots, m(l-1)-\lambda_2,m(l-1)-\lambda_1)$ is $m$-special, i.e.,  if and only if $(m(l-1)-\lambda_n,m(l-1)-\lambda_{n-1},\ldots, m(l-1)-\lambda_2)$ is $m$-special. However, applying the reciprocity principle once more, this is $m$-special if and only if $(\lambda_2,\lambda_3,\ldots,\lambda_n)$ is $m$-special.

\end{proof}

\q We now describe the principles of row  removal and node removal that will be used extensively  in Section 5.

\begin{proposition}   Let $n\geq 2$ and $m\geq 1$. If $\lambda=(\lambda_1,\ldots,\lambda_n)$ is an $m$-special (\resp.  $m$-good)  partition  then $(\lambda_1,\ldots,\lambda_{n-1})$ and $(\lambda_2,\ldots,\lambda_n)$ are $m$-special (\resp. $m$-good) partitions.

\end{proposition}

\begin{proof}  We give the argument for $m$-good. The $m$-special case is similar. We put $\mu=(\lambda_1,\ldots,\lambda_{n-1})$. Consider the natural module $E=E_n$ for $G(n)$. We have $E_n=E_{n-1}\oplus L$, where $L$ is the $K$-span of $e_n$ (and $E_{n-1}$ is the $K$-span of $e_1,\ldots,e_{n-1}$).  We regard $H=G(n-1)\times G(1)$ as a subgroup of $G(n)$, in the obvious way.  Then $E_n=E_{n-1}\oplus L$ is an $H$-module decomposition. Since $L(\lambda)$ is a composition factor of $S(E_n)^{\otimes m}$ it is a composition factor of $S^\alpha E_n$, for  some sequence $\alpha=(\alpha_1,\ldots,\alpha_m)$. The $H$-module $L(\lambda)$ has highest weight $\lambda$ and so has $L_{n-1}(\mu)\otimes L_1(\lambda_n)$ as a composition factor .

\q For  $r\geq 0$ we have $S^r(E)=\bigoplus_{r=r_1+r_2} S^{r_1}(E_{n-1}) \otimes S^{r_2} L$  as $H$-modules.  It follows that 
$L_{n-1}(\mu)\otimes L_1(\lambda_n)$ must be a composition factor of a module of the form $S^{u_1}(E_{n-1})\otimes\cdots\otimes  S^{u_m}(E_{n-1}) \otimes M$, for some $u_1,\ldots,u_m\geq 0$,  and one dimensional $G(1)$-module $M$.  Restricting to $G(n-1)$ gives that $\mu$ is $m$-good.

\q The result for $(\lambda_2,\ldots,\lambda_n)$ is obtained by restricting to $G(1)\times G(n-1)$ and arguing in the same way.

\end{proof}

\subsection*{Constrained  Modules and Node Removal}

We fix $m\geq 0$.  We say that a partition is $m$-constrained if it has at most $m$ parts.

\begin{definition}  Let $M$ be a finite dimensional polynomial module with a good filtration. We say that $M$ is $m$-constrained  if each $\lambda\in \Lambda^+(n)$ such that $(M:\nabla(\lambda))\neq 0$ is $m$-constrained.  We say that $M$ is $m$-deficient if  $(M:\nabla(\lambda))=0$ for every $m$-constrained  element $\lambda$ of $\Lambda^+(n)$. 

\end{definition}

\begin{remark} Note that if $M$ is a finite dimensional polynomial  module with a good filtration and character $\chi=\sum_{\lambda\in \Lambda^+(n)} r_\lambda \chi(\lambda)$ then $M$ is $m$-constrained if $\lambda$ is $m$-constrained whenever $r_\lambda\neq 0$ and $M$ is a $m$-deficient if $r_\lambda=0$ for all $m$-constrained $\lambda$.
\end{remark}

\begin{lemma}  Let $M$ be finite dimensional polynomial module with a good filtration and suppose that $M$ is $m$-deficient. Then for every finite dimensional polynomial module $V$ with a good filtration the polynomial module $M\otimes V$ is $m$-deficient.

\end{lemma}

\begin{proof} By the above remark it is enough to show that the coefficient of $\chi(\tau)$ in the character of $M\otimes V$ is zero for all $m$-constrained $\tau\in \Lambda^+(n)$.  It follows that it is enough to note that for $\lambda,\mu\in \Lambda^+(n)$ with $\lambda$  being $m$-constrained
 the coefficient of $\chi(\tau)$  in $\chi(\lambda)\chi(\mu)$ is $0$ for all $m$-constrained $\tau\in \Lambda^+(n)$.  So it is enough to show that for any symmetric function $\psi$ in $n$ variables $ \psi\chi(\lambda)$ is a $\zed$-linear combination of Schur symmetric functions $\chi(\tau)$ with $\tau$ not $m$-constrained. The ring of symmetric function is generated by the elementary symmetric functions $e_r=\chi(1^r)$, for $1\leq r\leq n$ so it enough to show that each $e_r\chi(\lambda)$ is a sum of terms $\chi(\tau)$, with $\tau$ not $m$-constrained. However, by Pieri's formula $e_r\chi(\lambda)$ is a sum of terms $\chi(\tau)$ where the diagram of $\tau$ is obtained by adding boxes to the diagram of $\lambda$, so the result is clear.

\end{proof}

\begin{lemma} Let $\lambda\in \Lambda^+(n)$. Then $\lambda$ is $m$-good if and only if $I(\lambda)$ is not $m$-deficient.

\end{lemma}

\begin{proof} We have that $\lambda$ is $m$-good if and only if there exists some $m$-constrained partition $\mu$ such that $[\nabla(\mu):L(\lambda)]\neq 0$, by Lemma 2.2.  By reciprocity, as in Section 1.2,  this is if and only if there exists an $m$-constrained partition $\mu$ such that $(I(\lambda):\nabla(\mu))\neq 0$, i.e,. if and only if $I(\lambda)$ is not $m$-deficient.

\end{proof}

\bf Definitions \rm

\q Let $\lambda$ be a partition.  

(i)  We call a node $R$ of $\lambda$ (or more precisely of the diagram of $\lambda$)    {\it removable} if the removal of $R$ from the diagram of $\lambda$ leaves the diagram of a partition, which will be denoted $\lambda_R$. Thus the node  $R$ is removable node if it has the form $(i,\lambda_i)$ for some $1\leq i\leq  \l(\lambda)$ and either $i=\l(\lambda)$ or 
$\lambda_i >\lambda_{i+1}$.

(ii)  An {\it addable} node $A$ of $\lambda$ is an element of $\nat\times \nat$ such that the addition of $A$ to the diagram of $\lambda$ gives the diagram of a partition, which will be denoted $\lambda^A$. Thus $A$ is addable if it has the form $(i,\lambda_i+1)$ for some $1\leq i\leq \l(\lambda)$ and either $i=1$ or $\lambda_i<\lambda_{i-1}$ or  $A=(\l(\lambda)+1,1)$. 

(iii)  The {\it residue} of a node $A=(i,j)$ of a partition $\lambda$ is defined to be the congruence class of  $j-i$ modulo $l$. 

(iv) Let $A$ and $B$ be removable or addable nodes of $\lambda$. We  shall say that  $A$ is {\it lower} than $B$ if  $A=(i,r)$, $B=(j,s)$ and $i>j$. 

(v) We say that a removable node  of $\lambda$ is {\it suitable} if its  residue  is different from the residue of  each lower addable node.

(vi) We say that a removable node $A=(i,\lambda_i)$ of $\lambda$ is {\it co-suitable} if the {\it transpose node} $A'=(\lambda_i,i)$ is a suitable node for $\lambda'$.

\medskip

\q Recall (or see \cite{Mac}, I, Section 1, Exercise 8) that partitions $\lambda$ and $\mu$ of the same degree have  the same $l$-core if and only if for each $0\leq r<l$ then number of nodes of $\lambda$ of residue $r$ is equal to the number of nodes of $\mu$ with residue $r$.

\medskip

\begin{lemma}    Suppose that $\lambda$ is a partition and $R=(h,\lambda_h)$ is a suitable node of $\lambda$. Then, for all $n$ sufficiently large,  we have: 

(i) $I_n(\lambda)$ is a direct summand of $I_n(\lambda_R)\otimes E_n$;

(ii)  if $I_n(\lambda_R)$ is $m$-deficient then so is $I_n(\lambda)$.

Furthermore if  $\lambda$ is $m$-good then so is $\lambda_R$.

\end{lemma}

\begin{proof}  We have an embedding of $\nabla_n(\lambda_R)$ in $I_n(\lambda_R)$ and hence an embedding of $\nabla_n(\lambda_R)\otimes E_n$ in $I_n(\lambda_R)\otimes E_n$. 
Let $\mu=\lambda_R$.  By Pieri's formula the character of $M=\nabla_n(\mu)\otimes E_n$ is the sum $\sum_A \chi(\mu^A)$, with $A$ running over all addable nodes of $\mu$ with $\l(\mu^A)\leq n$.  Thus we have $\ch M=\sum_i \chi(\mu+\ep_i)$, where the sum is over all $1\leq i\leq n$ such that $(i,\mu_i+1)$ is an addable node of $\mu$.  Thus $M$ has a good filtration $0=M_{n+1}\leq M_{n}\leq \cdots \leq M_1=M$, where $M_i/M_{i+1}$ is $\nabla_n(\mu+\ep_i)$ if $(i,\mu_{i+1})$ is addable, and $0$ otherwise.  Let $J=M_h$. Then $M_h/M_{h+1}$ is $\nabla_n(\lambda)$ and $\Ext^1_G(M_h/M_{h+1},M_{h+1})=0$,  since $M_{h+1}$ has a filtration with sections   $\nabla(\mu+\ep_i)$, with $i>h$ and 
$$\Ext^1_G(M_h/M_{h+1},\nabla(\mu+\ep_i))=\Ext^1_G(\nabla(\lambda),\nabla(\mu+\ep_i))=0$$
since $\lambda$ and $\mu+\ep_i$ have different cores and so the modules $\nabla(\lambda)$ and $\nabla(\mu+\ep_i)$ lie in different blocks.

\q  Hence $\nabla_n(\lambda)$ embeds in $I_n(\lambda_R)\otimes E_n$ and $I_n(\lambda_R)\otimes E_n$  is injective so that $I_n(\lambda_R)\otimes E_n$ contains the injective module $I_n(\lambda)$. Moreover if $I_n(\lambda_R)$ is $m$-deficient then by Lemma 3.10 $I_n(\lambda_R)\otimes E_n$ is $m$-deficient and so too is  $I_n(\lambda)$. This proves (i) and (ii). The final assertion follows from (ii) and Lemma 3.10.

\end{proof}


\bigskip\bigskip
\bigskip\bigskip


\section{Distinguished partitions and some Mullineux combinatorics}

\q   We shall assume some familiarity with the terminology of the Mullineux bijection, as explained in \cite{Mull}.   This   applies to the case in which   $l$ is prime but the combinatorics is in fact valid for $l$ arbitrary. A suitable reference for the more general context is \cite{JB}. 

\q The  length of the edge of the diagram of a  partition $\lambda$ is denoted $e(\lambda)$.

\q The length of the  $l$-edge (i.e.,  the sum of the lengths of the $l$-segments) will be denoted $e_l(\lambda)$. Recall that  $\P_\reg$ is the set of all $l$-regular partitions.  We recall that the Mullineux involution $\Mull:\P_\reg\to \P_\reg$  is defined recursively.  For $\lambda\in \P_\reg$ we call $\Mull(\lambda)$ its  Mullineux conjugate. The Mullineux conjugate of the empty set is the empty set. If $\lambda\in \P_\reg$ is not empty and $\nu$ is the partition whose diagram is obtained by removing the $l$-edge from the diagram of $\lambda$ then $\Mull(\lambda)$ is the unique $l$-regular partition such that the removal of the $l$-edge from the diagram of  $\Mull(\lambda)$ leaves the diagram of $\Mull(\nu)$ and

\begin{equation*}
\l(\m(\lambda))=
\begin{cases} 
e_l(\lambda)-\l(\lambda), & \text{if $l\mid e_l(\lambda)$};\cr
e_l(\lambda)-\l(\lambda)+1,  &\text{if $l\nmid e_l(\lambda)$.}
\end{cases}
\end{equation*}

\q An easy induction shows that if $e(\lambda)<l$ then $\Mull(\lambda)=\lambda'$. We shall use this property several times in what follows, without further reference.

\quad For a partition $\lambda$ it will be sometimes convenient to use the notation $\lambda=a_1^{t_1}a_2^{t_2}\dots$  to indicate that the entry $a_1$ appears $t_1$-times, $a_2$ appears $t_2$-times and so on.

\begin{definition} Let $0<m<l$. We say that a partition $\lambda$ is {\it $m$-distinguished}  if  $\lambda$ has the form  $\lambda^0+l\barlambda$, with $\lambda^0=(l-m)^ka_1\ldots a_m$, with $k\geq 0$, $l-m>a_1\geq \cdots\geq a_m\geq 0$ and $\barlambda$ a partition with $\barlambda_1<m$.
\end{definition}

\q Our approach is to describe the composition factors of a tensor product of truncated symmetric powers in terms of the distinguished partitions.

\bs

\bf Notation \,\rm   Let $0< m< l$. We write  $\Phi_m$ to be the set of partitions $\lambda=(\lambda_1,\lambda_2,\ldots)$ such that $\l(\lambda)\leq m$ and $\lambda_1-\lambda_m \leq l-m$.  

\q We note that  a partition $\lambda$ belongs to $\Phi_m$ if and only if we can write $\lambda=r^m+\alpha$, for some $r\geq 0$ and a partition $\alpha$ with $\alpha_1\leq l-m$, $\l(\alpha)<m$.

\begin{remark} It is easy to see that if $\mu$ is a (non-zero) $l$-regular partition with edge length at most $l$ then $\mu\in \Phi_m$, where $m=\l(\mu)$.
\end{remark}

\bs 

Our interest in this set of partitions is explained by the following result.

\begin{lemma} Let $0<m<l$.  A restricted partition $\lambda$ is $m$-distinguished  if and only if $\lambda'\in \Phi_{l-m}$.
\end{lemma}
\begin{proof} If $\lambda=(l-m)^ka_1\ldots a_m$ as above then  $\lambda=(l-m)^k\bigcup \nu$, where $\nu=a_1\ldots a_m$. Thus we have $\lambda'=((l-m)^k)'+ \nu'=k^{l-m}+\tau$, where $\tau=\nu'$, and this has the required form to qualify as an element of $\Phi_{l-m}$. Moreover the argument may be reversed and so the result holds.
\end{proof}

\begin{lemma}   Let $0<m<l$. If $\lambda\in \Phi_m$ then $e_l(\lambda)\leq l$.  
\end{lemma}

\begin{proof} We  write $\lambda=r^m+\alpha$ as above. If $r=0$ then we have $e(\lambda)=\alpha_1+\l(\alpha)-1<m+l-m-1<l$. If $r=1$ then we have $e(\lambda)=m+\alpha_1+1-1\leq m+l-m+1-1=l$. 

\q Now suppose $r>1$. Then we have $\lambda=(r-1)^m+\mu$, where $\mu=1^m+\alpha$. By the case just considered we have $e(\mu)\leq l$ so that the $l$-edge of $\mu$ has length at most $l$ and contains the node $(m,1)$ and hence the first $l$-segment of  $\lambda$ contains the node $(m,r)$. In particular the first $l$-segment contains a node from the final row of $\lambda$ and so there is only one $l$-segment, i.e., $e_l(\lambda)\leq l$.
\end{proof}

\begin{lemma} Let $0<m<l$. Let $\lambda\in \Phi_m$ and let $\mu$ denote the partition obtained by removing the $l$-edge of $\lambda$. Then we have $\mu\in\Phi_m$.
\end{lemma}

\begin{proof}  We write $\lambda$ in the form $r^m+\alpha$, as above. If $r=0$ then $e_l(\lambda)=e(\lambda)<l$ and $\mu$ is obtained by removing the entire edge of $\lambda$. The result is clear in this case.  Suppose now that $r>0$ but $e(\lambda)\leq l$.  Then $\lambda_1+m-1\leq l$ so that $\lambda_1\leq l-m+1$.  Since we remove the node $(1,\lambda_1)$ in obtaining $\mu$ we must have $\mu_1\leq l-m$.  Also, we remove  the entire final row of $\lambda$ in obtaining  $\mu$ so we must have $\l(\mu)<m$.  But now $\l(\mu)<m$ and $\mu_1\leq l-m$ gives $\mu\in \Phi_m$.

\q Now suppose $r>0$, $e(\lambda)>  l$ and so, by Lemma 4.4, $e_l(\lambda)=l$. Therefore the node $(m,1)$ does not belong to the $l$-edge of $\lambda$.  Thus we may write $\lambda=1^m+\nu$,  with $\nu\in \Phi_m$ and $\mu=1^m+\barnu$, where $\barnu$ is obtained by removing the $l$-edge from $\nu$.  We may assume inductively that $\barnu\in \Phi_m$ and hence $\mu=1^m+\barnu\in \Phi_m$.
\end{proof}

\begin{proposition} Let $0<m<l$.  The Mullineux correspondence restricts to a bijection $\Phi_m\to \Phi_{l-m}$.
\end{proposition}

\begin{proof} It suffice to show that $\Mull(\Phi_m)\subseteq \Phi_{l-m}$ since, replacing $m$ by $l-m$, we then get $\Mull(\Phi_{l-m})\subseteq \Mull(\Phi_m)$.  

\q Let $\lambda\in \Phi_m$  and let $\mu=\Mull(\lambda)$. We   write $\lambda=r^m+\alpha$, with $\l(\alpha)<m$, $\alpha_1\leq l-m$. 

\q First suppose that $r=0$. Then $e(\lambda)=\alpha_1+\l(\alpha)-1\leq l-m+m-1-1<l$ so that $\mu=\lambda'=\alpha'\in \Phi_{l-m}$.

\q Next suppose that $r>0$ but $e(\lambda)<l$.  Then $m+\lambda_1-1<l$ and $\mu=\lambda'$ so that $\mu_1=\l(\lambda)=m$ and $\l(\mu)=\lambda_1 \leq l-m$ so $\mu\in \Phi_{l-m}$. 

\q Now suppose that $r>0$, $e(\lambda)\geq l$ so that $e_l(\lambda)=l$.  Let $\barlambda$ be the partition obtained by removing  the $l$-edge from $\lambda$ and let $\theta=\Mull(\barlambda)$.  By Lemma 4.5  we have $\barlambda\in \Phi_m$ so  we can assume by induction that $\theta\in \Phi_{l-m}$, in particular $\l(\theta) \leq l-m$. 

\q We first consider the case in which $\l(\theta)<l-m$.  The $l$-edge of $\mu$ is the edge so that $l=\mu_1+l-m-1$, i.e., $\mu_1=m+1$.  We have $\l(\mu)=l-\l(\lambda)=l-m$ so that $\mu_{l-m}>0$ and $\mu_1-\mu_{l-m}\leq m$ so that that $\mu\in \Phi_{l-m}$.

\q It remains to consider the case $\l(\theta)=l-m$.  Then we may write $\theta=t^{l-m}+\phi$, with $t=\theta_{l-m}$ and we have that $\mu=t^{l-m}+\psi$, where $\psi$ is the 
partition with $\l(\psi)=l-m$ and such that the removal of the $l$-edge of $\psi$ leaves $\phi$.  We can assume inductively that $\psi\in \Phi_{l-m}$ so that $\mu=t^{l-m}+\psi\in \Phi_{l-m}$.
\end{proof}

\begin{corollary}  An  $l$-restricted  $m$-distinguished partition is $m$-special.

\end{corollary}

\begin{proof} Let $\lambda$ be a restricted $m$-distinguished partition. The we have $\lambda'\in \Phi_{l-m}$ by Lemma 4.3.  Hence we have $\Mull(\lambda')\in \Phi_m$, by Proposition 4.6  and hence $\l(\Mull(\lambda'))\leq m$ and so $\lambda$ is $m$-special by Proposition 3.2.

\end{proof}

\q We shall  prove a generalisation of this.

\begin{definition} Let now $\mu\in \P_\reg$. The sequence of Mullineux components $\mu^1,\mu^2,\dots$ of $\mu$ is  defined as follows. Suppose that the first $l$-segment of $\mu$ ends in the row $r_1$, the second in $r_2$ etc. Then  $\mu^1=(\mu_1,\dots,\mu_{r_1})$, $\mu^2=(\mu_{r_1+1},\dots,\mu_{r_2})$ etc.

\end{definition}

\q Note that, in the above situation, we have $\mu=\mu^1\bigcup \cdots \bigcup \mu^t$.  

\q We shall also develop an alternative notion, which will be useful in Section 5, to express $\mu$ in terms of its Mullineux components. 

\q  Let $\alpha,\rho$ be partitions with $\alpha\neq 0$. We shall say that the pair $(\alpha,\rho)$ is compatible if $\alpha_h\geq \rho_1$,where $h$ is the length of $\alpha$. If $(\alpha,\rho)$ is compatible we write $(\alpha|\rho)$ for the concatenation $(\alpha_1,\ldots,\alpha_h,\rho_1,\rho_2,\ldots)$.  For $k\geq 2$ and partitions    $\alpha^1,\alpha^2,\ldots,\alpha^{k+1}$,  such that  $\alpha^1,\ldots,\alpha^k\neq 0$ and such that the pair $(\alpha_i,\alpha_{i+1})$ is compatible for $1\leq i\leq k$  the concatenation $(\alpha^1|\alpha^2|\cdots|\alpha^{k+1})$ is defined  recursively by $(\alpha^1|\alpha^2|\cdots|\alpha^{k+1})=(\alpha^1| (\alpha^2|\cdots|\alpha^{k+1})).$

\q Thus, in particular, if $\mu$ is a $l$-regular partition with Mullineux components $\mu^1,\ldots,\mu^t$, as above, we have $\mu=(\mu^1|\mu^2|\cdots|\mu^t)$.

\q We note that the notion of Mullineux components easily extends to arbitrary partitions. Thus, for an  arbitrary partition $\lambda$ we write $\lambda=(\alpha|\rho)$, for partitions $\alpha,\rho$, where the first $l$-segment  of $\lambda$ has final node in the last row of $\alpha$, i.e., we have $\alpha=\lambda$ if $e_l(\lambda)\leq l$ and if $e_l(\lambda)\geq l$ then $\alpha=(\lambda_1,\ldots,\lambda_h)$, $\rho=(\lambda_{h+1},\ldots)$, where $h$ is minimal  such that $\lambda_1-\lambda_{h+1}+h\geq l$.  We then say that $\alpha$ is the first Mullineux component of $\lambda$ and say that the second  Mullineux component of $\lambda$ is the first Mullineux component of $\rho$, and so on. We shall use these components, in the general context,  in Section 5.

\begin{lemma}
Let $\mu\in \P_\reg$ have  Mullineux components $\mu^1,\mu^2,\dots,\mu^t$.  Then 
$$\l(\m(\mu))=\l(\m(\mu^1))+\dots+\l(\m(\mu^t)).$$

\end{lemma}

\begin{proof}
We have that $l\mid e_l(\mu)$ if and only if  $l\mid e_l(\mu^t)$. Therefore in  the case   $l\mid e_l(\mu)$  we get,
\begin{align*}
 \l(\m(\mu))=& e_l(\mu)-\l(\mu)=\sum^t_{i=1} e_l(\mu^i)-\sum^t_{i=1}\l(\mu^i)\cr
 =&\sum^t_{i=1}( e_l(\mu^i)-\l(\mu^i))=\sum^t_{i=1} \l(\m(\mu^i)).
 \end{align*}
For the case in which  $l\nmid e_l(\mu)$ and so $l\nmid e_l(\mu^t)$ we get 
\begin{align*}
 \l(\m(\mu))&= e_l(\mu)-\l(\mu)+1=\sum^t_{i=1} e_l(\mu^i)-\sum^t_{i=1}\l(\mu^i)+1\cr
 &=\sum^{t-1}_{i=1}( e_l(\mu^i)-\l(\mu^i))+( e_l(\mu^t)-\l(\mu^t)+1)\cr
 &=\sum^t_{i=1} \l(\m(\mu^i)).
 \end{align*}
\end{proof}

\begin{proposition}  Let $\lambda$ be a $l$-restricted partition and $m$ be a positive integer.  Then $\lambda$ is $m$-special if and only if it is possible to write
$$\lambda=\lambda^1+\cdots+\lambda^t$$
where $\lambda^i$ is a restricted   $m_i$-distinguished partition,  for $1\leq i\leq t$ and $m=m_1+\cdots+m_t$.

\end{proposition}

\begin{proof} Certainly any such a partition is $m$-special, by Corollary 4.7 and Lemma 2.4.

\q We now suppose  that $\lambda$ is $l$-restricted and $m$-special and show that it has  the required form. Thus $\l(\Mull(\lambda'))\leq m$ and it is clearly harmless to assume $\l(\Mull(\lambda'))=m$ (which we do). 

\q We write $\mu=\lambda'$ and consider the sequence of Mullineux components $\mu^1,\ldots,\mu^t$ of  $\mu$.  We define $m_i=\l(\Mull(\mu^i))$, $1\leq i\leq t$. Then we have $m=m_1+\cdots+m_t$, by Lemma 4.9. Moreover, we have
\begin{align*}\lambda&=\mu'=(\mu^1\bigcup\cdots \bigcup \mu^t)'\cr
&=\lambda^1+\cdots+\lambda^t
\end{align*}
where $\lambda^i=(\mu^i)'$, $1\leq i\leq t$.  Thus it suffices to prove that $\lambda^i$ is $m_i$-distinguished and so, by Lemma 4.3, it suffices to prove that $\mu^i\in \Phi_{l-m_i}$, $1\leq i\leq t$.

\q Suppose first that $1\leq i\leq t$ and $e_l(\mu^i)=l$.  (This is the case if $i<t$.)   We have $\l(\Mull(\mu^i))=l-\l(\mu^i)$ so that $\l(\mu^i)=l-m_i$.  Suppose that the $l$-edge of $\mu^i$ ends at the node $(l-m_i,k)$ then (by considering the diagram obtained by removing the first $k-1$ columns from the diagram of $\mu^i$) we see that the length of the $l$-edge of $\mu^i$ is
$$(\mu^i)_1-(k-1)+l-m_i-1$$
so that $(\mu^i)_1=k+m_i$ and $(\mu^i)_1-\mu_{l-m_i}\leq m_i$ and $\mu^i\in \Phi_{l-m_i}$, as required.

\q It remains to consider the case $i=t$ and $e_l(\mu^i)<l$. Then 
$$m_t=\l(\Mull(\mu^t))=\l((\mu^t)')=(\mu^t)_1.$$
Moreover we have 
$$\l(\mu^t)=e_l(\mu^t)-m_t-1<l-m_t$$
and $\mu^t\in \Phi_{l-m_t}$, as required.

\end{proof}

\q Finally we record a couple of results that will be needed in our treatment of composition factors.

\begin{lemma} Let $1<m<l$. If $\theta=(\theta_1,\ldots,\theta_m)\in \Phi_m$ then $(\theta_2,\ldots,\theta_m)\in \Phi_{m-1}$.
\end{lemma}

\begin{proof} Certainly $(\theta_2,\ldots,\theta_m)$ has length at most $m-1$. Also, we have  
$$\theta_2-\theta_m\leq \theta_1-\theta_m\leq l-m< l-(m-1).$$

\end{proof}

\q We now fix $n$ and consider the reflection with respect to $m$ of an $m$-distinguished partition.

\begin{lemma}  Let $\lambda$ be an $m$-distinguished partition and suppose that $n\geq \l(\lambda)$. Then $\daggerlambda$ is $m$-distinguished.
\end{lemma}

\begin{proof} We write 
$$\lambda=(l-m,\ldots,l-m,a_1,\ldots,a_m,0,\ldots,0)+l(\mu_1,\ldots,\mu_r,0,\ldots,0)$$
with $l-m>a_1\geq\cdots\geq a_m\geq 0$ and $m>\mu_1\geq \cdots \geq \mu_r\geq 0$.
Then we have
\begin{align*}\daggerlambda&=(m(l-1),m(l-1),\ldots,m(l-1))-l(0,\ldots,0,\mu_r,\ldots,\mu_1)\cr
&-(0,\ldots,0,a_m,\ldots,a_1,l-m,\ldots,l-m)\cr
&=(l-m,\ldots,l-m)-(0,\ldots,0,a_m,\ldots,a_1,l-m,\ldots,l-m) \cr
&+l(m-1,\ldots,m-1)-l(0,\ldots,0,\mu_r,\ldots,\mu_1)\cr
&=(l-m,\ldots,l-m,l-m-a_m,\ldots,l-m-a_1,0,\ldots,0)\cr
&+l(m-1,\ldots,m-1,m-1-\mu_r,\ldots,l-m-\mu_1,0,\ldots,0)
\end{align*}
which is $m$-distinguished.

\end{proof}

\begin{remark} From the definition of $\barS(E)$ we see that a $1$-special partition has first entry $\lambda_1\leq l-1$ and in particular $\lambda$ is restricted.  Hence, from Proposition 4.10, $\lambda$ has the form $(l-1,\ldots,l-1,b)$ (with $l-1\geq b\geq 0$).  Assume that $K$ has positive characteristic $p$.  Thus $\lambda$ is $1$-special for $\dotG(n)$ if it has the form $(p-1,\ldots,p-1,b)$.
From Corollary 2.7 we get that  every composition factor of $S(E)$ has the form 
$$L(\lambda^0)\otimes (\dotL(\lambda^1) \otimes \cdots\otimes \dotL(\lambda^t)^{\dotF^{t-1}})^F$$
where $\lambda^0$ is $1$-special for the group $G(n)$ and $\lambda^1\ldots,\lambda^t$ are $1$-special for the group $\dotG(n)$.  Specialising to the classical case $q=1$ we thus 
 recover the description of Krop, \cite{K}, and Sullivan, \cite{Sull}, describing the composition factors of symmetric powers of $E$.

\end{remark}


\bigskip\bigskip
\bigskip\bigskip


\section{Towards the Main Results}

\q We here assemble the final ingredients needed in the proofs of the main results. We first prove that if $\lambda$ is an $m$-distinguished restricted partition then there exists a $1$-distinguished partition $\alpha$ and an  $(m-1)$-distinguished restricted $\mu$ such that $L(\lambda)$ is a composition factor of $L(\alpha)\otimes L(\mu)$. We begin with a couple of  preliminary results, given in the next section.

\bs

\bf Notation\, \rm For non-negative integers $b_1,\ldots,b_m$ we write $Q(b_1,\ldots,b_m)$ for the partition obtained by arranging the numbers $b_1,\ldots,b_m$ in descending order.

\begin{lemma} Let $(a_1,\ldots,a_m)$ and $(b_2,\ldots,b_m)$ be partitions.  Suppose that $(b_2,\ldots,b_m)\geq (a_2,\ldots,a_m)$ and $(a_1,\ldots,a_m)\geq Q(a_1,b_2,\ldots,b_m)$. Then we have $(a_2,\ldots,a_m)=(b_2,\ldots,b_m)$.
\end{lemma}

\begin{proof} If $a_1\geq b_2$ then $Q(a_1,b_2,\ldots,b_m)=(a_1,b_2,\ldots,b_m)$ so we have \\
$(a_1,a_2,\ldots,a_m)\geq (a_1,b_2,\ldots,b_m)$
and hence $(a_2,\ldots,a_m)\geq (b_2,\ldots,b_m)$ and therefore $(a_2,\ldots,a_m) = (b_2,\ldots,b_m)$.

\q If $a_1<b_2$ then $Q(a_1,b_2,\ldots,b_m)$ has first entry $b_2$ and since $(a_1,\ldots,a_m)\geq Q(a_1,b_2,\ldots,b_m)$ we get $a_1\geq b_2$ and so   this case does not arise.
\end{proof}

\begin{remark}  Our interest  in the above is via Pieri's formula. Recall that for $a\geq 0$ and $\lambda\in \Lambda^+(n)$ the character of  the $G(n)$-module $\nabla(a)\otimes \nabla(\lambda)=S^a E\otimes \nabla(\lambda)$ is  $\sum_{\mu\in S}  \chi(\mu)$, where $S$ is the set of all partitions with at most $n$ part whose diagram may be obtained by adding a box to $a$ different columns of the diagram of $\lambda$, see \cite{Mac}, Chapter I, Section 5. Hence $S^a E\otimes \nabla(\lambda)$ has a good filtration with sections $\nabla(\mu)$, $\mu\in S$.   It  is not difficult to convince oneself that if $\l(\lambda)<n$ and  $\lambda=(b_2,\ldots,b_n)$  then the set $S$ has unique minimal element $Q(a,b_2,\ldots,b_n)$. Thus, in this situation, the module $V=\nabla(a)\otimes \nabla(b_2,\ldots,b_n)$ has a good filtration with $0=V_0<V_1<\cdots < V_t=V$ with $V_1=\nabla(Q(a,b_2,\ldots,b_n))$.

\end{remark}

\begin{proposition} Let $1<m<l$ and let $\lambda$ be a restricted $m$-distinguished partition of degree $r$ and let $n\geq r$. Then there exists a $1$-distinguished partition $\alpha$ and a restricted  $(m-1)$-distinguished partition $\mu$ such that  the $G(n)$-module $L(\lambda)$ is a composition factor of the $G(n)$-module  $L(\alpha)\otimes L(\mu)$.
\end{proposition}

\begin{proof}  We have  $\Mull(\lambda')=(a_1,\ldots,a_m)$, for some $(a_1,\ldots,a_m)\in \Phi_m$, by Lemma 4.3 and Proposition 4.6.
The partition $(a_1,\ldots,a_m)$ is the unique maximal element of the set $\{\tau \in \Lambda^+(n,r) \vert [\nabla(\tau):L(\lambda)]\neq 0\}$, by Proposition 1.5.2.   The module  $\nabla(a_1,\ldots,a_m)$  occurs as a section in a good filtration of $\nabla(a_1)\otimes \nabla(a_2,\ldots,a_m)$ (see the above Remark), in particular we have  $[\nabla(a_1)\otimes \nabla(a_2,\ldots,a_m):L(\lambda)]\neq 0$ and hence $[\nabla(a_1)\otimes L(\theta):L(\lambda)]\neq 0$ for some $\theta\in \Lambda^+(n)$  such that $L(\theta)$ is a composition factor of $\nabla(a_2,\ldots,a_m)$. By Lemma 3.1, $\theta$ is restricted. 

\q Let $\phi=\Mull(\theta')$.  Then,  by Proposition 1.5.2, $\phi$ is the unique maximal element of the set $\{\tau\in \Lambda^+(n,s) \vert [\nabla(\tau):L(\theta)]\neq 0\}$, where $s=\deg(\theta)$, in particular we have $\phi\geq (a_2,\ldots,a_m)$ and so $\l(\phi)\leq m-1$.  We write $\phi=(b_2,\ldots,b_m)$.

\q Since $[\nabla(a_1)\otimes L(\theta):L(\lambda)]\neq 0$, we have 
$$[\nabla(a_1)\otimes \nabla(b_2,\ldots, b_m):L(\lambda)]\neq 0.$$

Let $\xi\in \Lambda^+(n)$ be such that $\nabla(\xi)$ occurs as a section in a  good filtration of $\nabla(a_1)\otimes \nabla(b_2,\ldots, b_m)$  and $[\nabla(\xi):L(\lambda)]\neq 0$.  Thus  $(a_1,\ldots,a_m)\geq \xi$ so, by the Remark 5.2, we have $\xi\geq Q(a_1,b_2,\ldots,b_m)$.

\q Now we have  $(a_1,\ldots,a_m)\geq \xi \geq Q(a_1,b_2,\ldots,b_m)$ and by  Lemma 5.1,  we have $(a_2,\ldots,a_m)=(b_2,\ldots,b_m)$, i.e., $\phi=(a_2,\ldots,a_m)$.  By Lemma 4.11  we have $\phi\in \Phi_{m-1}$ and hence, by Lemma 4.3 and Proposition 4.6, $\theta$ is a restricted $(m-1)$-distinguished partition. Now $L(\lambda)$ is a composition factor of $L(\alpha)\otimes L(\theta)$ for some composition factor $L(\alpha)$ of $\nabla(a_1)$. By Lemma 3.1, $\alpha$ is restricted and hence of the form $(l-1,\ldots,l-1,b)$, i.e., a restricted $1$-distinguished partition.

\end{proof}

\q We now turn our attention to an analysis of the $l$-edge of a partition.

\bs

{\bf Definitions and Notation}

\medskip

Let $\lambda$ be a partition.

(i)  We will denote  the $l$-edge of a partition $\lambda$ by $\E_l(\lambda)$.

(ii) We will say that $\lambda$ is {\it edge $l$-connected}  if the collection of nodes $\E_l(\lambda)$ is connected. More precisely,  if $\lambda$  has (non-zero) Mullineux components $\lambda^1,\lambda^2,\ldots,\lambda^{t+1}$ then $\lambda$ is edge $l$-connected if  for each $1\leq i\leq t$ the final node of the first $l$-segment of $(\lambda^i|\lambda^{i+1})$ lies directly above the node $(\l(\lambda^i)+1,\lambda^{i+1}_1)$.  This may be also expressed by the condition 

\bs

(*) $\lambda^i_1-\lambda^{i+1}_1+\l(\lambda^i)=l$, for $1\leq i< t$. 

\bs

\q Now write  $\lambda=(\alpha|\rho)$, where $\alpha$ is the first Mullineux component. We say that $\lambda$ is {\it initially edge $l$-connected} if either $\rho=0$ or $\rho\neq 0$ and $\alpha_1-\rho_1+\l(\alpha)=l$. Thus $\lambda$ is edge $l$-connected if and only if it is initially edge $l$-connected and $\rho$ is edge $l$-connected.

\q  If $\lambda$ is not  edge $l$-connected we will say that it is {\it edge $l$-disconnected}.

(iii)    If $\lambda$ is a partition and $H$ is a skew $l$-hook in the diagram of $\lambda$ (as in \cite{James}, Chapter 17) we denote by $\lambda_H$ the partition whose diagram is obtained by removing $H$ from the diagram of $\lambda$.

\begin{lemma}  Let $\lambda$ be an edge $l$-connected  partition such that  $e_l(\lambda)$ is not divisible by  $l$.

(i)   If $H$ is any skew $l$-hook of (the diagram of) $\lambda$ then $\lambda_H$ is edge $l$-connected,  $e_l(\lambda)$ not divisible by  $l$ and $(\lambda_H)_1=\lambda_1$.

(ii) We have $\core(\lambda)_1=\lambda_1$.

\end{lemma}

\begin{proof}  (i)   If $e(\lambda)<l$ then the result is vacuously true. We assume now that $\lambda$ is a counterexample of minimal degree.  Thus  we can write   $\lambda=(\alpha|\rho)$, with $\alpha$ the first  Mullineux component of $\lambda$ and $\rho\neq 0$.   Let $h=\l(\alpha)$.  If no node of $H$ belongs to the first $h$ rows then we may write $\lambda_H=(\alpha|\rho_J)$, for some skew $l$-hook $J$ of $\rho$. But then $\rho_J$ is edge $l$-connected, $e_l(\rho)$ is  not divisible by  $l$ and $(\rho_J)_1=\rho_1$, by minimality.  But then  the same holds for $\lambda$. Hence $H$ contains  a node of the diagram of $\alpha$.

\q Now the number of nodes in the part of the edge from $(1,\lambda_1)$ to $(h,\rho_1)$ is the edge length of $(\lambda_1-\rho_1+1,\ldots,\lambda_h-\rho_1+1)$ i.e., $\lambda_1-\rho_1+1+h-1=l$. So if $H$ involves $(1,\lambda_1)$ then it ends in $(h,\rho_1)$. But this is impossible since then the removal of $H$ from the diagram of $\lambda$ would not result in   the diagram of a partition. Hence  $H$ does not contain the node $(1,\lambda_1)$. Similarly, $H$ can not be contained entirely within the diagram of $\alpha$.  Thus $H$ contains the nodes $(h,\rho_1)$ and $(h+1,\rho_1)$.  

\q Let $\beta$ denote the first  Mullineux component of $\rho$ and write $\rho=(\beta|\sigma)$, so that $\lambda=(\alpha|\beta|\sigma)$. Let $k=\l(\beta)$. We note that $H$ is contained within the diagram of $(\alpha|\beta)$. This is of course true if $\sigma=0$. For  $\sigma\neq 0$ we would  otherwise  have that $H$ contains the node $(h,\beta_1)$ and also  nodes $(h+1,\lambda_{h+1})$ and $(h+k+1,\sigma_1)$ and hence would contain more nodes than are in the edge of $(\lambda_{h+1}-\sigma_1+1,\ldots,\lambda_{h+k+1}-\sigma_1+1)$  and we would have 
$$l>  \lambda_{h+1}-\sigma_1+1+k-1=    \beta_1-\sigma_1+k=l.$$  

\q Let $\mu=\lambda_H$.   Thus we have $\mu_1=\lambda_1$ and $\mu_h=\rho_1-1$ (since $H$ contains the nodes $(h,\rho_1)$ and $(h+1,\rho_1)$.   Now we have 
$$\mu_1-\mu_h+(h-1)= \lambda_1-\rho_1+h=l.$$
 Hence $\mu$ has first Mullineux component $\gamma=(\mu_1,\ldots,\mu_{h-1})$ of length $h-1$ and $\mu$ is initially $l$-connected.   Let $\de=(\mu_h,\ldots,\mu_{h+k})$ so now $\mu=(\gamma|\delta|\sigma)$. 
 
 \q If $\sigma\neq 0$ then since 
 $$\delta_1-\sigma_1+(k+1)=\mu_h-\sigma_1+k+1=\beta_1-1-\sigma_1+k+1=l$$
  we have that $\delta$ is the second Mullineux component of $\mu$ so that 
  $$e_l(\lambda_H)=e_l(\gamma|\delta|\sigma)=l+l+e_l(\sigma)=e_l(\alpha)+e_l(\beta)+e_l(\sigma)=e_l(\lambda)$$
  and we are done.
  
  \q So we can assume that $\sigma=0$, i.e., $\lambda=(\alpha|\beta)$ and $e(\beta)<l$. But now we have
   $$e(\delta)\leq \delta_1+(k+1)-1=\rho_1-1+k=\beta_1 - 1+k\leq e(\beta)$$
   and again we are done unless $e(\delta)=0$, i.e.,  unless $\delta=0$.  In that case we have $0=\delta_1=\mu_h=\rho_1-1$, so  $\rho_1=1$.   We then have 
   $$0<e(\mu)=e(\gamma)\leq \mu_1+(h-1)-1=\lambda+h-\rho_1-1=l-1$$
   and the proof is complete.

(ii) This follows by repeated application of (i).

\end{proof}

\begin{lemma} Let $\lambda$ be an  $l$-regular partition. Assume that $\lambda$ is  edge $l$-connected  and $l\mid e_l(\lambda)$. Let $\tilde\lambda$ be the partition whose diagram is obtained by removing the first column from the diagram of  $\lambda$. Then $\l(\m(\lambda))=\l(\m(\tilde\lambda))$.

\end{lemma}

\begin{proof} We write $\lambda=(\alpha|\rho)$, where $\alpha$, of length $h$, say, is the first  Mullineux component of $\lambda$.  Note that the first $l$-segment does not end at $(h,1)$, for otherwise $\rho$ would have the form $(1^s)$, and  we would have $e_l(\lambda)=e_l(\alpha)+s$ divisible by $l$, which is incompatible with the $l$-regularity of $\lambda$. Thus $\tilde\lambda=(\tilde\alpha|\tilde\rho)$ (where the diagram of $\tilde\alpha$ (\resp.   $\tilde\rho$)  is obtained by removing the first column of the diagram of $\alpha$ (\resp.  $\rho$)). So we get 
\begin{align*}\l(\m(\lambda))&=\l(\m(\alpha))+\l(\m(\rho))\cr
&=\l(\m(\tilde\alpha))+\l(\m(\tilde\rho))=\l(\m(\tilde\lambda))
\end{align*}
by induction on degree.

\end{proof}

\begin{remark} Suppose $(\alpha,\beta,\gamma)$  is a compatible triple of partitions (i.e., $(\alpha,\beta)$ and $(\beta,\gamma)$ are compatible pairs)  and $\lambda=(\alpha|\beta|\gamma)$.  Let $B$ be an addable node of $\beta$ and let $A$ be the corresponding addable node of $\lambda$, i.e.., the node such that 
 $\lambda^A=(\alpha|\beta^B|\gamma)$. Let $S$ be an edge node of $\beta$ and let $R$ be the corresponding edge node of $\lambda$, i.e., if $S=(i,j)$ then $R=(\l(\alpha)+i,j)$.  Then $\res(A)=\res(R)$ if and only if $\res(B)=\res(S)$.

\end{remark}

\q In order to prove our final two  lemmas we  need one more  useful remark.  

\begin{remark}
 Let $\lambda$ be an  $l$-regular partition with $e_l(\lambda)\leq l$. We embed $\lambda$ into the rectangular partition  $\mu=(\lambda_1)^{\l(\lambda)}$. For any node $B=(i,j)$ of the skew diagram $[\mu]\backslash [\lambda]$ we have 
 $\res(B)\neq\res(1,\lambda_1)=\lambda_1-1$. Indeed, since $e_l(\lambda)\leq l$, we have that $\lambda$ has  only one $l$-segment and if $(\l(\lambda),c)$ is the last node of $\E_l(\lambda)$, then $\lambda_1+\l(\lambda)-c\leq l$. Thus we have  $2\leq i\leq \l(\lambda)$ and  $c<j \leq\lambda_1$. If $\res(B)=\lambda_1-1$, then we would have that $\lambda_1-j +(i-1)=0\ \mod\  l$. But this is impossible since,

$$1\leq \lambda_1-j+(i-1)<\lambda_1+\l(\lambda)-c\leq l.$$

\end{remark}

\begin{lemma}

Let $\mu$ be an $l$-regular partition. Assume that $\mu$ is edge $l$-disconnected.  Then there is  a co-suitable node $R$ of $\mu$ such that $\mu_R$ is $l$-regular and  $\l(\m(\mu))
=\l(\m(\mu_R))$.

\end{lemma}

\begin{proof}    Assume not and that $\mu$ is a counterexample of minimal degree.  Our strategy is to first work up from the point in the  diagram at which connectedness first fails  to show in particular  that $\mu_1=\mu_2$ and then work down from the top of the diagram using this information. 

\q We write $\mu=(\mu^1|\mu^2|\cdots|\mu^m)$, where $\mu^1,\ldots,\mu^m$ are the (non-zero)\\
 Mullineux components. Let $h_i=\l(\mu^i)$, for $1\leq i\leq m$. We suppose that $k$ is minimal such that $(\mu^1|\cdots|\mu^{k+1})$ is edge $l$-disconnected. Thus we have: 
\bs

\sl (*) $\mu^i_1-\mu^{i+1}_1+h_i=l$, for $1\leq i <k$ and $\mu^k_1-\mu^{k+1}_1+h_k>l$.

\rm 
\bs

\q We write $R_i$ for the node $(\sum_{j<i} h_j +1,\mu^i_1)$, for $1\leq i\leq k$. 
\bs

{\it Step 1.}\q  \sl We have $\res(R_1)=\res(R_2)=\cdots=\res(R_k)$.

\rm

\bs
{\it Proof of Step 1.}   For $1\leq i<k$ we have $\mu^i_1-\mu^{i+1}_1+h_i=l$ so that $\mu^i_1-\sum_{j<i} h_j-1$ is congruent (modulo $l$) to $\mu^{i+1}_1-h_{i}-\sum_{j<i} h_j-1=\mu^{i+1}_1-\sum_{j<i+1}h_j-1$, i.e., $\res(R_i)=\res(R_{i+1})$.

\bs

{\it Step 2.}  \sl We have $\mu^k_1=\mu^k_2$.
\rm

\bs

{\it Proof of Step 2.}   Suppose for a contradiction that we have $\mu^k_1>\mu^k_2$.   Then $R_k$ is a removable node. We claim that $R_k$ is co-suitable. If not let $A$ be an addable node above $R_k$ whose residue is that of $R_k$.

 \q  We have 
 $\mu^A=(\mu^1|\cdots|\mu^{i-1}|(\mu^i)^B|\mu^{i+1}|\cdots|\mu^m)$ for some $1\leq i<k$ and some addable node $B$ of $\mu^i$.   Now $R_k$ has the same residue as $R_i$, by Step 1. Let $S$ be the corresponding node of $\mu^i$, i.e.,  $S=(1,\mu^i_1)$. If $R_k$ and $A$ have the same residue then so do  $S$ and $B$, by Remark 5.6. This is obviously not true if $B=(1,\mu^i_1+1)$ and also  impossible for $B\neq(1,\mu^i_1+1)$ by Remark 5.7. Hence $R_k$ is co-suitable.

 \q Now we have $\mu_{R_k}=(\tau^1|\cdots|\tau^m)$ where $\tau^i=\mu^i$ for $1\leq i<k$, $\tau^k=(\mu_1^k-1,\mu^k_2,\ldots,\mu^k_{h_k})$ and $\tau^i=\mu^i$ for $i>k$. Moreover, it is easy to check that $\tau^1,\tau^2,\ldots,\tau^m$ are the Mullineux components of $\mu_{R_k}$. We also have that $\l(\Mull(\tau^i))=\l(\Mull(\mu^i))$ for all $i$ so that   
$$\l(\Mull(\mu_{R_k}))=\sum_{i=1}^m  \l(\Mull(\tau^i))=\sum_{i=1}^m  \l(\Mull(\mu^i))=\l(\Mull(\mu))$$
and we have a contradiction.

 \bs

 {\it Step 3.}  \sl We have $\mu^i_1=\mu^i_2$, for $1\leq i\leq k$.
 
 \rm\bs

 {\it Proof of Step 3.}   Assume not and   that $s$ is such that $\mu^s_1>\mu^s_2$ but $\mu^i_1=\mu^i_2$ for all $s<i\leq k$. 
We consider the node $R=R_s$. By  the argument of Step 2,  $R$ is co-suitable.

\q Let $a=\mu^s_1$. We claim that $\mu_R$ is $l$-regular. If not then the $(h_1+\cdots+h_{s-1}+1)$th row in the diagram of $\mu$ is followed by $l-1$ rows of length $a-1$.   In particular we have $\mu^s=a(a-1)^{h_s-1}$. Therefore $h_s=l-1$. Moreover, we have $\mu^s_1-\mu^{s+1}_1+h_s=l$ so $\mu^{s+1}_1=a-1$. Hence we have 
$\mu^{s+1}_1=\mu^{s+1}_2=a-1$.  But then $a-1$ is the length of the $l$ rows following the $(h_1+\cdots+h_{s-1}+1)$th row in the diagram of $\mu$. But $\mu$ is $l$-regular so this is impossible and the claim is established. 

\q Now we have $\mu_R=(\tau^1|\cdots|\tau^m)$ where $\tau^i=\mu^i$ for $1\leq i<s$, $\tau^s=(\mu_1^s-1,\mu^s_2,\ldots,\mu^s_{h_s},\mu^{s+1}_1)$, $\tau^i=(\mu^i_2,\ldots,\mu^i_{h_i},\mu^{i+1}_1)$, for $s<i<k$, $\tau^k=(\mu^k_2,\ldots,\mu^k_{h_k})$ and $\tau^i=\mu^i$ for $i>k$. Moreover, it is easy to check that $\tau^1,\tau^2,\ldots,\tau^m$ are the Mullineux components of $\mu_R$. We also have that $\l(\Mull(\tau^i))=\l(\Mull(\mu^i))$ for $i\neq s,k$, $\l(\Mull(\tau^s))=\l(\Mull(\mu^s))-1$ and 
 $\l(\Mull(\tau^k))= \l(\Mull(\mu^k))+1$. Therefore,  
$$\l(\Mull(\mu_R))=\sum_{i=1}^m  \l(\Mull(\tau^i))=\l(\Mull(\mu))$$
and we have a contradiction. 

\bs

{\it Step 4.} \sl Conclusion

\rm 

\bs

\q Let $R$ be a removable node of $\mu$ such that $\mu_R=(\mu^1_S|\mu^2|\cdots|\mu^m)$,  for a removable node $S$ of $\mu^1$. Then by Step  3 we have that $R\neq(1,\mu_1)$. It is easy to check that $R$ is co-suitable and if $\mu_R$ is an $l$-regular partition we have  

 \begin{align*}\l(\Mull(\mu_R))&=\l(\Mull(\mu^1_S))+\l(\Mull(\mu^2))+\dots+\l(\Mull(\mu^m))\cr
 &=\l(\Mull(\mu^1))+\l(\Mull(\mu^2))+\dots+\l(\Mull(\mu^m))\cr
 &=\l(\Mull(\mu)).
 \end{align*}
 Hence, we may assume that $\mu_R$ is not $l$-regular  so we have   $\mu^1=a^u(a-1)^{l-1-u}$, for some $2\leq u\leq l-1$ and $\mu^2=(a-1)^u\mu^2_{u+1}\dots\mu^2_{h_2}$ with $\mu^2_{u+1}<a-1$. 
 
 \q Consider  the  node $R=(l-1+u,a-1)$. This is removable and has residue $a-u$. Moreover the addable nodes above $R$ are  $(1,a+1)$ and $(u+1,a)$ and these have  residues $a$ and $a-u-1$. Hence $R$ is co-suitable. Suppose that $\mu_R$ is $l$-regular. Then we have that $\mu_R=(\mu^1|\mu^2_S|\cdots|\mu^m)$, where $S=(u,a-1)$ and   again \begin{align*}\l(\Mull(\mu_R))&=\l(\Mull(\mu^1))+\l(\Mull(\mu^2_S))+\dots+\l(\Mull(\mu^m))\cr
 &=\l(\Mull(\mu^1))+\l(\Mull(\mu^2))+\dots+\l(\Mull(\mu^m))\cr
 &=\l(\Mull(\mu)).
 \end{align*}
Therefore we must have $\mu^2=(a-1)^u(a-2)^{l-1-u}$ and $\mu^3=(a-2)^u\mu^3_{u+1}\dots\mu^3_{h_3}$ with $\mu^3_{u+1}<a-2$. Continuing in this way,  we may assume that $\mu=(\mu^1|\mu^2|\cdots|\mu^m)$ with $\mu^i=(a-i+1)^u(a-i)^{l-1-u}$ for $1\leq i\leq k-1$ and $\mu^k=(a-k+1)^u\mu^k_{u+1}\dots\mu^k_{h_k}$ and $2\leq u \leq l-1$. 
 
 \q Consider finally  the node $R=((k-1)(l-1)+u,a-k+1)$. This is removable and has residue $a-u$. Moreover the addable nodes above $R$ are  $(1,a+1)$ and $((i-1)(l-1)+u+1,a-i+1)$,  for $1\leq i\leq k-1$,  with residues $a$ and $a-u+1$. Hence $R$ is co-suitable. In addition $\mu_R$ is $l$-regular. We have  $\mu_R=(\mu^1|\cdots|\mu^k_S|\dots|\mu^m)$, where $S=(u,a-k+1)$ and  
 
  \begin{align*}\l(\Mull(\mu_R))&=\l(\Mull(\mu^1))+\dots+\l(\Mull(\mu^k_S))+\dots+\l(\Mull(\mu^m))\cr
 &=\l(\Mull(\mu^1))+\dots+\l(\Mull(\mu^k))+\dots+\l(\Mull(\mu^m))\cr
 &=\l(\Mull(\mu)).
 \end{align*}
 
 Thus $\mu$ is not a counterexample and the proof is complete.

\end{proof}

It will be of great importance, especially for the proof of Lemma 5.9, to review the proof of Lemma 5.8  and give an explicit description of the removable node $R$ we obtain  with the properties of Lemma 5.8.

\bs

\q Let $\mu$ be an $l$-regular partition which is $l$-disconnected. Then by Lemma 5.8 we have that there is a co-suitable node $R$ of $\mu$ such that $\mu_R$ is  $l$-regular and $\l(\m(\mu))=\l(\m(\mu_R))$. By the proof of Lemma 5.8 we have that the node $R$ is obtained in one of two different ways, depending on the shape of $\mu$. We describe explicitly the two situations here. We write $\mu=(\mu^1|\mu^2|\cdots|\mu^m)$ where $\mu^1,\dots,\mu^m$ are the (non-zero) Mullineux components. Let $h_i=\l(\mu^i)$ for $1\leq i\leq m$. Let $k$ be minimal such that $(\mu^1|\cdots|\mu^{k+1})$ is edge $l$-disconnected. Thus we have $\mu^i_1-\mu^{i+1}_1+h_i=l$, for $1\leq i <k$ and $\mu^k_1-\mu^{k+1}_1+h_k>l$.
\bs

{\it Case 1.} Assume  that there is some  $1\leq i\leq k$ with $\mu^i_1>\mu^i_2$. Let  $s=\max\{i\ |\mu^i_1>\mu^i_2, 1\leq i\leq k\}$. Then we have that the co-suitable node $R$ with the above properties is the node $R=R_s=(\sum_{j<s} h_j +1,\mu^s_1)$.

\bs

{\it Case 2.} Assume  that $\mu^i_1=\mu^i_2$ for all $1\leq i\leq k$. Let $t$ be the minimal value of $1\leq i\leq k$ with the property that the Mullineux component $\mu^t$ has a removable node, say $T$, such that if $R$ is the removable node of $\mu$ with $\mu_R=(\mu^1|\mu^2|\cdots|\mu^t_T|\cdots|\mu^k|\cdots|\mu^m)$ then $\mu_R$ is $l$-regular. The existence of this node is guaranteed by the fact that $\mu$ is  $l$-disconnected. In this case it follows that, for some $a$, we have  $\mu^j=(a-j+1)^u(a-j)^{l-u-1}$ for $1\leq j\leq t-1$ and $\mu^t=(a-t+1)^u\mu^t_{u+1}\dots\mu^t_{h_t}$ for some $2\leq u\leq l-1$. Moreover the node $R=((t-1)(l-1)+u,a-t+1)$ is the co-suitable node we obtain  with the desired properties. 

\bigskip\bigskip

{\bf Some further  Definitions, Notations and Remarks}

\medskip

 (i) A {\it weakly addable} node for a partition $\lambda$ is an element of $\nat\times \nat$ which has the form, $(i, \lambda_i+1)$ for some $1\leq i\leq \l(\lambda)$ or $(\l(\lambda)+1,1)$. Observe, that an addable node of $\lambda$ is always a weakly addable node.
 
(ii) Let $\lambda$ be a partition. We write $\lambda$ as usual in the form $\lambda=\lambda^0+l\bar\lambda$, with $\lambda^0$ be $l$-restricted. Let $A=(i,\lambda_i+1)$ be an addable node for $\lambda$ with $1\leq i\leq \l(\lambda^0)+1$. We consider now   $A^0=(i,\lambda^0_i+1)$. This is a weakly addable node for $\lambda^0$ and the nodes $A$ and $A^0$ have the same residue. We will refer to $A^0$ as the weakly addable node of $\lambda^0$ corresponding to the addable node $A$ of $\lambda$.  

 (iii) Let $\lambda=\lambda^0+l\bar\lambda$ be a non-restricted partition. Let $A_0=(i,\lambda^0_i)$ be a removable node of $\lambda^0$ such that $\lambda^0_{A_0}$ is  a restricted partition. Then $A=(i,\lambda_i)$ is a removable node for $\lambda$ and $\lambda_A=\lambda^0_{A_0}+l\bar\lambda$.
 Moreover $A$ and $A_0$ have the same residue.
 
 \begin{lemma}
 
 Let $\lambda=\lambda^0+l \bar\lambda$ be a non-restricted partition with $\l(\bar\lambda)\leq\l(\lambda^0)$. Let $\mu=(\lambda^0)'$. Assume that $\mu$ is edge $l$-disconnected. Then there is  a suitable node $S=(i,\lambda_i)$ of $\lambda$ such that:
 
 (i) the  node $S_0=(i,\lambda^0_i)$ is a suitable node of $\lambda^0$ and $\lambda^0_A$ is $l$-restricted;  and 
 
(ii) the node $R=(\lambda^0_i,i)$ is a co-suitable node of $\mu$ such that  $\mu_R$ is  $l$-regular and $\l(\m(\mu))=\l(\m(\mu_R))$.

 \end{lemma}
 
 \begin{proof}
 
 We will produce the node $S$ of $\lambda$ with the above properties using the co-suitable nodes of $\mu$ described in Lemma 5.8. 
 
 \q There is a co-suitable node $R$ of $\mu$ such that $\mu_R$ is  $l$-regular and\\
  $\l(\m(\mu))=\l(\m(\mu_R))$. By the discussion following the proof of Lemma 5.8 we may produce $R$ according to one of the cases below.

 \q We write $\mu=(\mu^1|\mu^2|\cdots|\mu^m)$ where $\mu^1,\dots,\mu^m$ are 
 the (non-zero) \\
 Mullineux components. Let $h_i=\l(\mu^i)$ for $1\leq i\leq m$. Let $k$ be minimal such that $(\mu^1|\cdots|\mu^{k+1})$ is edge $l$-disconnected. Thus we have $\mu^i_1-\mu^{i+1}_1+h_i=l$, for $1\leq i <k$ and $\mu^k_1-\mu^{k+1}_1+h_k>l$.
 
 \bs

{\it Case 1.} Assume that $\mu^i_1>\mu^i_2$ for some $1\leq i\leq k$ and $s=\max\{i\ |\mu^i_1>\mu^i_2, 1\leq i\leq k\}$. Then we have that $R=R_s=(\sum_{j<s} h_j +1,\mu^s_1)$. We consider first the transpose node $S_0=(\mu^s_1,\sum_{j<s} h_j +1)$ of $\lambda^0$. Since $R_s$ is co-suitable and $\mu_{R_s}$ is $l$-regular we have that $S_0$ is suitable and $\lambda^0_{S_0}$ is $l$-restricted. We take now the node $S=(\mu^s_1,\sum_{j<s} h_j +1+l \bar\lambda_{\mu^s_1})$ of $\lambda$. The node $S$ is removable. Hence, it remains to prove that is also suitable. We assume for a contradiction that it is  not. Then there is an addable node, say $U=(r,\lambda_r)$, of $\lambda$ below $S$ with the same residue as  $S$. Since $\l(\bar\lambda)\leq\l(\lambda^0)$ we can take now the corresponding weakly addable node $U^0=(r,\lambda^0_r+1)$ of $\lambda^0$. We have that $\res(U^0)=\res(U)$. 

\q We consider now the transpose node $V=(\lambda^0_r+1,r)$. We have that $\res(V)=\res(R_s)$. Moreover, since $U^0$ is a weakly addable node of $\lambda^0$ appearing lower than $S_0$ we get that  $V$ can only have one of the following forms:  $V=(1,\mu_1+1)$;  $V=(\sum_{j<i} h_j +1,\mu^i_1+k)$ for some $1< i\leq s$ with $1\leq k\leq \mu^{i-1}_{h_{i-1}}-\mu^i_1$ ; or $V=(\sum_{j<i} h_j +\ell,\mu^i_\ell+k)$ for some $1\leq i<s$ and $2\leq \ell \leq h_i$ with $1\leq k\leq \mu^{i}_{\ell-1}-\mu^i_\ell$. 

\q We can exclude directly the case $V=(1,\mu_1+1)$ because in this case $V$  is an addable node of $\mu$ and since  $\res(V)=\res(R_s)$, this contradicts the fact that $R_s$ is co-suitable.

\q Let $V=(\sum_{j<i} h_j +1,\mu^i_1+k)$ for some $1< i\leq s$ with $1\leq k\leq \mu^{i-1}_{h_{i-1}}-\mu^i_1$. We compare the residue of $V$ with the residue of the node $R_i=(\sum_{j<i} h_j +1,\mu^i_1)$. By Step 1 of the proof of Lemma 5.8 we have that $\res(R_i)=\res(R_s)$ and so $\res(V)=\res(R_i)$. Therefore, we get that  $\mu^i_1-\sum_{j<i} h_j -1$ is $\mu^i_1+k-\sum_{j<i} h_j -1 \ \mod \ l$. Thus $k$ must be congruent to $0 \ \mod\  l$.  However, this  is not the case   since $\mu^{i-1}_{h_{i-1}}-\mu^i_1< \mu^{i-1}_1-\mu^{i}_1+h_{i-1}=l$ and so $1\leq k<l$. Therefore we have a contradiction.

\q We have now the final case where $V=(\sum_{j<i} h_j +\ell,\mu^i_\ell+k)$ for some $1\leq i<s$ and $2\leq \ell \leq h_i$ with $1\leq k\leq \mu^{i}_{\ell-1}-\mu^i_\ell$. We compare the residue of $V$ with the residue of the node $R_i=(\sum_{j<i} h_j +1,\mu^i_1)$. Since $\res(R_i)=\res(R_s)$ we get that $\res(V)=\res(R_i)$. In particular  we deduce that the nodes $(1,\mu^i_1)$ and $(\ell,\mu^i_\ell+k)$ have the same residue. This contradicts the Remark 5.7. Therefore we have that the node $S$ is a suitable node for $\lambda$. 

\bs

\q We examine now the situation where the node $R$ is obtained from the second form of the partition $\mu$ as described in the remarks following Lemma 5.8.

\bs

{\it Case 2.} In this case we have that $\mu^i_1=\mu^i_2$ for all $1\leq i\leq k$. Let $t$ be the minimal value of $1\leq i\leq k$ with the property that the Mullineux component $\mu^t$ has a removable node, say $T$, such that if $R$ is the removable node of $\mu$ with $\mu_R=(\mu^1|\mu^2|\cdots|\mu^t_T|\cdots|\mu^k|\cdots|\mu^m)$, then $\mu_R$ is $l$-regular. Then, for some $a$,  we have  $\mu^j=(a-j+1)^u(a-j)^{l-u-1}$ for $1\leq j\leq t-1$ and $\mu^t=(a-t+1)^u\mu^t_{u+1}\dots\mu^t_{h_t}$ for some $2\leq u\leq l-1$ and the node $R=((t-1)(l-1)+u, a-t+1)$ is the co-suitable node of $\mu$ with the properties of Lemma 5.8. We consider the transpose node $S_0=(a-t+1, (t-1)(l-1)+u)$ of $\lambda^0$. Since $R$ is co-suitable and $\mu_{R}$ is $l$-regular  $S_0$ is suitable and $\lambda^0_{S_0}$ is $l$-restricted.  We take now the node $S=(a-t+1,(t-1)(l-1)+u+l\bar\lambda_{a-t+1})$ of $\lambda$. The node $S$ is removable. Hence, it remains to prove that is also suitable. We assume for contradiction that is not. Then there is an addable node, say $U=(r,\lambda_r)$, of $\lambda$ below $S$ with the same residue with $S$. Since $\l(\bar\lambda)\leq\l(\lambda^0)$ we can take now the corresponding weakly addable node $U^0=(r,\lambda^0_r+1)$ of $\lambda^0$. We have that $\res(U^0)=\res(U)$. 

\q We consider now the transpose node $V=(\lambda^0_r+1,r)$. We have  $\res(V)=\res(R)$. Moreover, since $U^0$ is a weakly addable node of $\lambda^0$ appearing lower than $S_0$ we get that  $V$ can only have one of the following forms:  $V=(1,a+1)$;  or $V=((j-1)(l-1)+u+1,a-j+1)$ for $1\leq j\leq t-1$. Here the node $V$ is always an addable node of $\mu$ and so $\res(V)=\res(R)$ contradicts the fact that $R$ is co-suitable. Therefore we deduce again that the node $S$ of $\lambda$ is a suitable node and the proof is complete.

\end{proof}

\q We finish this section with a Remark which follows immediately from  Lemma 5.4.

\begin{remark}

Let $\lambda=(\lambda_1,\lambda_2,\dots,\lambda_m)$ be an $l$-restricted partition. The simple module $L(\lambda)$ appears as composition factor of  $\nabla(\mu)$, for some partition $\mu$  with $\l(\mu)<m$ if and only if $\l(\core(\lambda))<m$.

\end{remark}

\begin{proof}
Assume first that there is a partition $\mu$ with $\l(\mu)<m$ and $[\nabla(\mu):L(\lambda)]\neq0$. Since $\core(\lambda)=\core(\mu)$ we get immediately that 

$$\l(\core(\lambda))=\l(\core(\mu))\leq \l(\mu)<m.$$

Assume now that $[\nabla(\mu):L(\lambda)]=0$ for every partition $\mu$ with $\l(\mu)<m$. Then by Proposition 3.2 we get that $\l(\m(\lambda'))=m$. Therefore, by the formula for $\l(\m(\lambda'))$ given in the beginning of section 4 we have that 

\begin{equation*}
m=
\begin{cases} 
e_l(\lambda')-\lambda_1, & \text{if $l\mid e_l(\lambda')$};\cr
e_l(\lambda')-\lambda_1+1,  &\text{if $l\nmid e_l(\lambda')$.}
\end{cases}
\end{equation*}

The first case gives that $e_l(\lambda')=\lambda_1+m>e(\lambda')$ which of course is impossible. Hence, we have that only the second case is possible. Therefore, $e_l(\lambda')=m+\lambda_1-1=e(\lambda')$ and $l\nmid e_l(\lambda')$. Thus, $\lambda$ is edge $l$-connected  and $l\nmid e_l(\lambda')$. Therefore by Lemma 5.4 we have that $\core(\lambda')_1=m$ and so $\l(\core(\lambda))=m$. 

\end{proof}


\bigskip\bigskip
\bigskip\bigskip


\section{The Main Results on  Composition Factors}

\begin{proposition} Let $m$ be a positive integer. Suppose that a partition $\lambda$ can be written in the form $\lambda=\lambda(1)+\cdots+\lambda(s)$, where $\lambda(i)$ is an $m_i$-distinguished (not necessarily restricted) partition and $m=m_1+\cdots+m_s$. Then $\lambda$ is $m$-special. 

\end{proposition}

\begin{proof}  Lemma 2.4 immediately reduces considerations to the case $s=1$. So we assume that $\lambda$ is $m$-distinguished. We write $\lambda=\lambda^0+l\barlambda$ for $\lambda^0,\barlambda\in \Lambda^+(n)$ with $\lambda^0$ restricted and $m$-distinguished and $\barlambda_1<m$. From Corollary 4.7 we may assume that $\barlambda\neq0$. Suppose that $\barlambda_1<m-1$.  Then $L(\lambda^0)$ is a composition factor of $L(\alpha)\otimes L(\mu)$, for $\alpha,\mu\in \Lambda^+(n)$, where  $\alpha$ is restricted,   $1$-distinguished and $\mu$ is  restricted and $(m-1)$-distinguished, by Proposition 5.3. But now, $\mu+l\barlambda$ is $(m-1)$-distinguished so, by induction on $m$, we have that $L(\mu+l\barlambda)=L(\mu)\otimes \dotL(\barlambda)^F$ is a composition factor of $\barS(E)^{\otimes (m-1)}$ and hence $L(\alpha)\otimes L(\mu)\otimes \dotL(\barlambda)^F$ appears as a section of $\barS(E)^{\otimes m}$. But $L(\alpha)\otimes L(\mu)\otimes \dotL(\barlambda)^F$  has a section $L(\lambda^0)\otimes \dotL(\barlambda)^F$ and so $L(\lambda)$ is a composition factor of $\barS(E)^{\otimes m}$. 

\q Thus we may assume  $\barlambda_1=m-1$. Assume that $\l(\barlambda)=  \l(\lambda)$. Then $\lambda$ has final entry at least $l$.  We consider reciprocity with respect to $n=\l(\lambda)$. 
Now 
$\mu=\daggerlambda$ has first entry at most $m(l-1)-l=(m-1)l-m$. Moreover,  $\mu$ is $m$-distinguished, by Lemma 4.12  and writing $\mu=\mu^0+l\barmu$,  for partitions $\mu^0,\barmu$, with $\mu^0$ restricted, we have  $\barmu_1<m-1$. Hence, by the case already considered, $\mu$ is $m$-special and hence by Lemma 3.4, $\lambda$ is $m$-special. 

\q So we now suppose $\l(\barlambda)=r<\l(\lambda)$.  If $\lambda^0_1=l-m$ then $\lambda_1=l-m+l(m-1)=m(l-1)$ so again by reciprocity   with respect to $n=\l(\lambda)$ we obtain a partition $\mu=\lambda^\dagger$ of shorter length. By induction on length we may assume that  $\mu$ is $m$-special and hence, by Proposition 3.5, $\lambda$ is $m$-special. 

\q Thus we may assume that $\l(\lambda)\leq m$ and  $\lambda=(a_1,\ldots,a_m)+l\barlambda$, with $l-m>a_1\geq \cdots\geq a_m\geq 0$ and $\barlambda_1=m-1$, $\barlambda=(\barlambda_1,\ldots,\barlambda_r)$, $0<r<m$.  We set $\nu=\barlambda-\omega_r$.  Then we have

\begin{align*}\lambda&=(l+a_1,\ldots,l+a_r,a_{r+1},\ldots,a_m)+l\nu\cr
&=(l-m+r,\ldots,l-m+r,a_{r+1},\ldots,a_m)\cr
&+(a_1+m-r,\ldots,a_r+m-r)  +l\nu.
\end{align*}

Now $\nu_1=m-2$ so we can write $\nu=\alpha+\beta$ for partitions $\alpha,\beta$ with $\alpha_1=m-r-1$, $\beta_1=r-1$.  Then 
$(l-m+r,\ldots,l-m+r,a_{r+1},\ldots,a_m)+l\alpha$ is $(m-r)$-distinguished and hence $(m-r)$-special and $(a_1+m-r,\ldots,a_r+m-l)+l\beta$ is $r$-distinguished  and hence $r$-special. Hence $\lambda$ is the sum of an  $(m-r)$-special and an  $r$-special partition   and hence,  by Lemma 2.4,  is $m$-special.

\end{proof}

\begin{proposition}  Let $\lambda$ be a partition and write $\lambda=\lambda^0+l\barlambda$, for partitions $\lambda^0,\barlambda$, with $\lambda^0$ restricted.  If $\lambda$ is $m$-good then $\lambda^0$ is $m$-good.

\end{proposition}

\begin{proof}

If  not, let  $\lambda=\lambda^0+l\barlambda$ be a counterexample of minimal degree.  We have that $\l(\bar\lambda)\leq \l(\lambda^0)$.  We see this in the following way. Let $\lambda=(\lambda_1,\lambda_2,\dots,\lambda_n)$ and suppose  $\l(\lambda^0)<\l(\bar\lambda)$. Let $\hat\lambda=(\lambda_1,\lambda_2,\dots,\lambda_{n-1})$. By Proposition 3.6 $\hat\lambda$ is $m$-good and by the minimality of the degree of $\lambda$ we get that $\hat\lambda^0$ is $m$-good. Since $\l(\lambda^0)<\l(\bar\lambda)$ we have that $\hat\lambda^0=\lambda^0$ and so $\lambda^0$ is $m$-good and  $\lambda$ is not a  counterexample. Therefore we assume from now on that $\l(\bar\lambda)\leq \l(\lambda^0)$.

\q We have  $\l(\lambda^0)\geq m+1$, for example by Lemma 2.2. We set  $\mu=(\lambda^0)'$. Hence, $\mu$ is an  $l$-regular partition with $\mu_1\geq m+1$. Moreover, since $\lambda^0$ is not $m$-good, we get by the Proposition 3.3,  that $\l(\m(\mu))\geq m+1$.

\medskip

{\it Case  1. }    Assume that $\E_l(\mu)$ is connected and that $l\nmid e_l(\mu)$. Then by Lemma  5.4 we have that $\core(\mu)_1=\mu_1\geq m+1$ and so $\l(\core(\lambda^0))\geq m+1$. But then $\l(\core(\lambda))\geq m+1$, contradicting the fact that $\lambda$ is $m$-good.

\medskip

{\it Case 2.  }   Assume now that $\E_l(\mu)$ is connected and $l\mid e_l(\mu)$. Let $\tilde\lambda$ be the partition obtained from $\lambda$ by  first row removal. Then by Proposition 3.6 we have that $\tilde\lambda$ is $m$-good and by the minimality of degree we have that $\tilde\lambda^0$ is $m$-good. Let $\tilde\mu$ the transpose of $\tilde\lambda^0$. Hence $\tilde\mu$ is the partition obtained from $\mu$ after removing the first column. Therefore we get by Lemma 5.5 that $\l(\m(\tilde\mu))=\l(\m(\mu))\geq m+1$, contradicting the fact that $\tilde\lambda^0$ is $m$-good.

\medskip

{\it Case 3. }  Therefore, we are left with the case in which  $\E_l(\mu)$ is disconnected.  By Lemma 5.9 there is a suitable node $S=(i,\lambda_i)$ of $\lambda$ with the following properties: $S_0=(i,\lambda^0_i)$ is a suitable node of $\lambda^0$; $\lambda^0_{S_0}$ is $l$-restricted; the node $R=(\lambda^0_i,i)$ is a co-suitable node of $\mu$; $\mu_R$ is an $l$-regular partition and; $\l(\m(\mu_R))=\l(\m(\mu))\geq m+1$. We consider these three nodes here. Since $S$ is a suitable node of $\lambda$ we have that $\lambda_S$ is $m$-good by Lemma 3.11. We write $\lambda_S=\lambda^0_{S_0}+l\barlambda$. By the minimality of the degree of $\lambda$ we get that $\lambda^0_{S_0}$ is $m$-good. We have that $(\lambda^0_{S_0})'=\mu_R$. Hence, $\l(\m(\mu_R))\leq m$ by Proposition 3.2. Therefore we have  a contradiction.

\end{proof}

\begin{corollary}
Let $\lambda$ be a partition and write $\lambda=\lambda^0+l\barlambda$, for partitions $\lambda^0,\barlambda$, with $\lambda^0$ restricted.  If $\lambda$ is $m$-special then $\lambda^0$ is $m$-special.

\end{corollary}

\begin{proof}

Let $\lambda$ be $m$-special, then it is $m$-good and so by Proposition 6.2 we have that $\lambda^0$ is $m$-good. By Proposition 3.2 we have then that $\lambda^0$ is also $m$-special.

\end{proof}

\begin{proposition}    Let $m$ be a positive integer.  Let $\lambda$ be a partition and write $\lambda=\lambda^0+l\barlambda$ for partitions $\lambda^0,\barlambda$, with $\lambda^0$ restricted.  Suppose $\lambda^0$ is $m$-special and $\lambda_1\leq m(l-1)$. Then $\lambda$ can be written in the form $\lambda=\lambda(1)+\cdots+\lambda(s)$, where $\lambda(i)$ is an $m_i$-distinguished (not necessarily restricted) partition and $m=m_1+\cdots+m_s$.  In particular $\lambda$ is $m$-special.

\end{proposition}

\begin{proof}     By an admissible pair of sequences for a partition  $\mu$ we mean a sequence 
  $(k_1,\ldots,k_t)$ of positive integers whose sum is $m$ and a sequence  
$(\mu(1),\ldots,\mu(t))$ of partitions whose sum is $\mu$ and  such that $\mu(i)$ is \\
$k_i$-distinguished, for $1\leq i\leq t$.  Less formally, we shall say that $\mu=\mu(1)+\cdots+\mu(t)$ is an admissible expression for $\mu$. We shall write $\mu(i)^0$ for the restricted part and $\barmu(i)$ for the non-restricted part, i.e., $\mu(i)^0$ and $\barmu(i)$ are partitions, with $\mu(i)^0$ restricted,  such that $\mu(i)=\mu(i)^0+l\barmu(i)$, for $1\leq i\leq t$.

\q Suppose that the result  is false and that $\lambda$ is a partition of minimal degree for which it fails. By Proposition 4.10, $\lambda$ is not restricted, i.e, $\barlambda\neq 0$. .  We choose $r>0$ such that $\barlambda-\omega_r$ is a partition. We put 
$$\mu=\lambda-l\omega_r=\lambda^0+l\barmu$$
where $\barmu=\barlambda-\omega_r$. By minimality, $\mu$ is writable in the required form. 

\bs

\it Step 1.   \rm  If   $\mu=\mu(1)+\cdots+\mu(t)$ is an admissible expression for $\mu$, with $\mu(i)$ a $k_i$-distinguished partition, for $1\leq i\leq t$, then $\barmu(i)_1=k_i-1$ for  $1\leq i\leq t$.
\bs

\it Proof of Step 1. \rm If not then for some $j$ we have $\barmu(j)_1<k_j-1$. Now putting 
$$\lambda(i)=\begin{cases} \mu(i), & {\rm if}\q  i\neq j;\cr
\mu(j)+l\omega_r,  & {\rm  if}\q  i=j
\end{cases}$$
we have that each $\lambda(i)$ is $k_i$-distinguished and $\lambda=\lambda(1)+\cdots+\lambda(t)$, contrary to assumption. 

\bs

\it Step 2.  \rm If   $\mu=\mu(1)+\cdots+\mu(t)$ is an admissible expression for $\mu$, with $\mu(i)$ a $k_i$-distinguished partition, for $1\leq i\leq t$ then we have $\mu(i)^0_1<l-k_i$ for  at least two values of $i$ (with $1\leq i\leq t$).

\bs

\it Proof of Step 2.  \rm If not then, after reordering, we can assume that $\mu(i)^0_1=l-k_i$ for $1\leq i\leq t-1$ so we get 
\begin{align*}\lambda_1=&l+\mu_1\cr
\geq & l+ (l-k_1)+l(k_1-1)+\cdots+ (l-k_{t-1})+l(k_{t-1}-1)+  l(k_t-1)\cr
&=l+(t-1)l-(m-k_t)+lm-lt=m(l-1)+k_t
\end{align*}
contrary to the fact that $\lambda_1\leq m(l-1)$.

\bs

 \q From now on we take $t$ to be minimal such that there exists an admissible expression  $\mu=\mu(1)+\cdots+\mu(t)$.
 
 \bs
 
 \it Step 3. \rm\,  There exists an admissible expression  $\mu=\mu(1)+\cdots+\mu(t)$   with  $\mu(i)^0=0$ for some $1\leq i\leq t$.
 
 \bs
 
 \it Proof of Step 3.  \rm \, Given an admissible pair $S$, say,  for $\mu$,  consisting of the sequence  
 $(k_1,\ldots,k_t)$ of positive integers (whose sum is $m$) and sequence $(\mu(1),\ldots,\mu(t))$ or partitions, we define  ${\rm index}(S)$ to be the minimum of the set $\{k_i\vert 1\leq i\leq t, \mu(i)^0_1< l-k_i\}$.  We consider admissible pairs  for $\mu$ whose index $h$,  say,  is as small as possible.  For such an admissible pair $S$ we define the defect  $d(S)$ to be the minimum of the set $\{\mu(i)^0_1 \vert \mu(i)^0_1< l-k_i, k_i=h\}$.  We further assume that $S$ is such that the defect of $S$ is as small as possible.  If $d(S)=0$ then we are done so we assume that $S$ has positive defect. 
 
 \q We arrange the terms in the admissible expression $\mu=\mu(1)+\mu(2)+\cdots+\mu(t)$ such that $k_1=h$, $d(S)=\mu(1)^0_1$ and $\mu(2)^0_1<l-k_2$ (using Step 2).  Note  that $k_1\leq k_2$ by minimality of the index.  We choose $u>0$ such that $\mu(1)^0-\omega_u$ is a partition.   Now by the definition of distinguished and the fact that $\mu(1)^0_1<l-k_1$ we have that $u\leq k_1$. Since $\mu(2)^0_1<l-k_2$ and $k_1\leq k_2$ we have that $\mu(2)+\omega_u$ is $k_2$-distinguished.  But now,  setting
 $$\nu(i)=\begin{cases} \mu(1)-\omega_u, & \hbox{ if } i=1;\cr
 \mu(2)+\omega_u, & \hbox{ if } i=2;\cr
 \mu(i), & \hbox{ otherwise.}
 \end{cases}
 $$
 we obtain an expression $\mu=\nu(1)+\cdots+\nu(t)$ and the corresponding admissible pair $T$,  say, has index equal to the index of $S$ (namely $h$) and smaller defect, a contradiction.

 \bs
 
 \it Step 4. \rm \, Conclusion.
 
 \bs
 
\q \rm We write $\mu=\mu(1)+\cdots+\mu(t)$ as in Step 3 and arrange the  numbering so that $\mu(1)^0=0$ and $\mu(2)^0_1<l-k_2$. If $k_1=1$ then $\bar\mu(1)_1=0$ so that $\mu(1)=0$, contradicting the minimality of $t$. Thus we have $k_1>1$. We choose $u>0$ such that $\barmu(1)-\omega_u$ is a partition. Then $\mu(1)-l\omega_u$ is $(k_1-1)$-distinguished and $\mu(2)+l\omega_u$ is $(k_2+1)$-distinguished. Moreover, we have that 
 $$\mu=(\mu(1)-l\omega_u)+(\mu(2)+l\omega_u)+\mu(3)+\cdots+\mu(r)$$
 is a an admissible expression for $\mu$.  Continuing in this way, we can find an admissible expression as above with  $k_1=1$,  a contradiction. 
 
 \bs
 
 \q The proof that $\lambda$ may be written in the required form is complete. We get that $\lambda$ is $m$-special from Proposition 6.1.

\end{proof}

\q We now put together Propositions 6.1 and 6.4 and Corollary 6.3 to give the main result of the paper.

\begin{theorem} Let $m$ be a positive integer. A partition $\lambda$ is $m$-special if and only if it can be written in the form $\lambda=\lambda(1)+\cdots+\lambda(s)$, where $\lambda(i)$ is an $m_i$-distinguished (not necessarily restricted) partition and $m=m_1+\cdots+m_s$.

\end{theorem}

\begin{proof}  A partition that is writable in the above form is $m$-special by Proposition 6.1.  Suppose now that $\lambda$ is $m$-special. Then, writing $\lambda=\lambda^0+l\barlambda$ for partitions $\lambda^0$, $\barlambda$, with $\lambda^0$ restricted, we have that $\lambda^0$ is $m$-special, by Corollary 6.3. Moreover, since the simple module $L(\lambda)$ is a composition factor of $\barS(E)^{\otimes m}$, we have $\lambda_1\leq m(l-1)$. Hence $\lambda$ is writable in the required form, by Proposition 6.4.

\end{proof}

\begin{corollary}

Let $m$ be a positive integer. A partition $\lambda$ is $m$-special if and only if it is $m$-good and $\lambda_1\leq m(l-1)$.

\end{corollary}

\begin{proof}

If $\lambda$ is $m$-special then it is $m$-good and $\lambda_1\leq m(l-1)$.  We now assume  that $\lambda$ is $m$-good and that $\lambda_1\leq m(l-1)$. We write $\lambda$ as $\lambda=\lambda^0+l\bar\lambda$ for partitions $\lambda^0,\barlambda$ with $\lambda^0$ restricted. Then, by Proposition 6.2 we have that $\lambda^0$ is $m$-good and so $m$-special by Proposition 3.2. Hence, since $\lambda_1\leq m(l-1)$, we get by Proposition 6.4 that $\lambda$ is $m$-special.

\end{proof}

\q We now consider  $\lambda\in \Lambda^+(n)$ and apply the above in the case $m=n$.

\begin{corollary}  Let  $\lambda\in \Lambda^+(n)$. Then $\lambda_1\leq n(l-1)$ if and only if  we can write $n$ as a sum of positive integers $n_1,\ldots,n_s$ and  $\lambda=\lambda(1)+\cdots+\lambda(s)$ with $\lambda(i)\in \Lambda^+(n)$ an $n_i$-distinguished partition, for $1\leq i\leq s$.
\end{corollary}

\begin{proof}  Clear from Proposition 6.4 and Corollary 6.6.
\end{proof}

\q  Let $G=G(n)$ and let $E$ be the natural module.  An element $\lambda=(\lambda_1,\ldots,\lambda_n)$ of $\Lambda^+(n)$ is called $n(l-1)$-bounded if $\lambda_1\leq n(l-1)$.  For such a partition $\lambda$ we define the truncated tensor product $\barS^\lambda(E)=\barS^{\lambda_1}(E)\otimes \cdots \otimes \barS^{\lambda_n}(E)$. It is clear that if $\mu\in \Lambda^+(n)$ is such that $L(\mu)$ is a composition factor of $\barS^\lambda(E)$ then $\mu_1\leq n(l-1)$, i.e., $\mu$ is $n(l-1)$-bounded.  One therefore obtains a square matrix of decomposition numbers $([\barS^\lambda(E):L(\mu)])$, with $\lambda,\mu$ running over $n(l-1)$-bounded partitions in $\Lambda^+(n)$.  Doty conjectures (in the classical situation)  that this matrix is non-singular, see \cite{Mar}, Conjecture 4.2.11. We note that  each $L(\mu)$ with $\mu$ an $n(l-1)$-bounded partition, appears as the composition factor of some $\barS^\lambda(E)$ - so at least the decomposition matrix conjectured to be non-singular contains no column consisting entirely of zeros.

\begin{corollary}  Let $G=G(n)$ and let $E$ be the natural module.  For each $n(l-1)$-bounded element $\mu$ of $\Lambda^+(n)$ there exists an $n(l-1)$-bounded partition $\lambda\in \Lambda^+(n)$ such that $[\barS^\lambda(E):L(\mu)]\neq 0$.

\end{corollary}

\begin{proof} Combining Theorem 6.5 and Corollary 6.7 we have that $[\barS(E)^{\otimes n}:L(\mu)]\neq 0$. Moreover, we have $\barS(E)=\bigoplus_{j=0}^{n(l-1)}\barS^j( E)$. Hence $\barS(E)^{\otimes n}$ is a direct sum of modules of the form $\barS^{j_1}(E)\otimes \cdots\otimes \barS^{j_n}(E)$, for some $0\leq j_1,\ldots,j_n\leq n(l-1)$. Such a module is has the character of  $\barS^\lambda(E)$, for some $n(l-1)$-bounded element $\lambda$ of $\Lambda^+(n)$. Hence we have $[\barS^\lambda(E):L(\mu)]\neq 0$, for some $n(l-1)$-bounded element $\lambda$ of $\Lambda^+(n)$,

\end{proof}

\section*{Acknowledgement}

The second author gratefully acknowledges the financial  support of EPSRC Grant EP/L005328/1.


\bigskip\bigskip
\bigskip\bigskip



\begin{thebibliography}{99}










\bibitem{JB} {J. Brundan, \emph{Modular branching rules and the Mullineux map for Hecke algebras of type A},  Proc. London Math. Soc. (3) {\bf 77}  (1998),   551-581.}

\bibitem{Cox} {A. G. Cox, \emph{The blocks of the $q$-Schur algebra}, J. Algebra {\bf 207} (1998), 306-325.}

\bibitem{CPS} {E. Cline, B. Parshall ,  L. Scott and W. van der Kallen, \emph{Rational and generic cohomology}, Invent. Math. {\bf 39}, (1977), 143-163.}



\bibitem{D4}{S. Donkin, \emph{Rational Representations of Algebraic Groups: Tensor Products and Filtrations}, Lecture Notes in Math.  1140, Springer  1985.}


\bibitem{D0} {S. Donkin, \emph{On Schur Algebras and Related Algebras II}, J.   Algebra, {\bf 111},  354-364,  1987.}




\bibitem{D1} {S. Donkin, \emph{On tilting modules for algebraic groups},  Math. Z. {\bf 212}, 39-60, 1993.}


\bibitem{D2} {S. Donkin, \emph{On Schur Algebras and Related Algebras IV: The Blocks of the Schur Algebras}, J.   Algebra, {\bf 168}, 400-429, 1994.}

\bibitem{DStd}{S. Donkin, \emph{Standard Homological Properties  for Quantum  ${\rm GL}_n$}, J. Algebra, {\bf 181}, 235-266, 1996.}


\bibitem{D3} {S. Donkin, \emph{The $q$-Schur algebra},  LMS Lecture Notes 253, Cambridge University Press 1998.}



\bibitem{DDV}{S. Donkin and M.  De Visscher, \emph{On projective and injective polynomial modules},  Math. Z.  {\bf 251}, 333-358,  2005.}


\bibitem{DG1} {S. Donkin and H. Geranios, \emph{Endomorphism Algebras of Some Modules for Schur Algebras and Representation Dimension},  Algebras and Representation Theory, Volume 17, Issue 2, 623-642, 2014.}


\bibitem{DG2} {S. Donkin and H. Geranios, \emph{Polynomially and Infinitesimally Injective Modules},  J.   Algebra {\bf 392},  125-141,  2013.}


\bibitem{DG3} {S. Donkin and H. Geranios, \emph{Composition Factors of  Tensor Products of  Symmetric Powers},   Journal of Algebra {\bf  438}, 24-47, 2015.}



\bibitem{DG4} {S. Donkin and H. Geranios,  \emph{Invariants of Specht modules},  Journal of Algebra, {\bf 439}, 188-224, 2015.}

\bibitem{DW}{S. R. Doty and G. Walker, \emph{Modular symmetric functions and irreducible modular representations of general linear groups}. J.    Pure and Applied Algebra, {\bf 82}, 1-26, 1992.}




\bibitem{EGS}{K. Erdmann, J. A. Green and M. Schocker , \emph{Polynomial Representations of ${\rm GL}_n$,  Second Edition with an Appendix on Schenstead Correspondence and Littelmann Paths}, Lecture Notes in Mathematics 830, Springer 2007.}

\bibitem{FK} {B. Ford and A. S. Kleshchev, \emph{A Proof of the Mullineux Conjecture}, Math. Z. {\bf 226}, 267-308,  1997.}


\bibitem{H} {D. J. Hemmer,  \emph{A Row Removal Theorem for the $\Ext^1$ Quiver of Symmetric Groups and Schur Algebras}, Proc. Amer. Math. Soc. Vol. {\bf 133}, No. {\bf 2}, 403-414,  2004.}

\bibitem{James}  {G. D. James,  \emph{The Representation Theory of the Symmetric Groups}, Lecture Notes in Mathematics 682, Springer 1970.}

\bibitem{Jan}{Jens Carsten Jantzen, \emph{Representations of Algebraic Groups}, second ed., Math. Surveys Monogr., vol 107, Amer. Math. Soc., 2003.}










\bibitem{K}{L. Krop, \emph{On the representations of the full matrix semigroup on homogeneous polynomials, I}. J.   Algebra, {\bf 99}, 284-300, 1986.}

\bibitem{Mac}{I. G. Macdonald, \emph{Symmetric Functions and Hall Polynomials}, 2nd Ed., Oxford Mathematical Monographs, Oxford University Press 1998.}

\bibitem{Mar} {S. Martin, \emph{Schur Algebras and Representation Theory}, Cambridge Tracts in Mathematics 
  11,  Cambridge University Press 1993.}

 \bibitem{Mull}{G.  Mullineux,\emph{ Bijections of $p$-regular partitions and $p$-modular irreducibles of the symmetric groups}, J. London Math.
  Soc.,  (2) {\bf20}, 60-66, 1979.}
 
 
 \bibitem{Sull}{J. B. Sullivan, \emph{Some representation theory for the modular general linear groups}, J. Algebra {\bf 45}, 516-535, 1977.}


\end{thebibliography}
\end{document}